\documentclass[11pt,a4paper,reqno]{amsart}

\usepackage{amssymb,amsmath}
\usepackage{amsthm}
\usepackage{latexsym}
\usepackage[colorlinks=true,linktocpage=true,pagebackref=false, citecolor=black,linkcolor=black]{hyperref}
\usepackage{graphics}
\usepackage{euscript}
\usepackage{mathrsfs}
\usepackage[dvips]{curves}
\usepackage{tikz}
\usetikzlibrary{arrows,decorations.pathmorphing,backgrounds,positioning,fit,matrix}
\usepackage{comment}
\usepackage{multirow}
\usepackage{xcolor}
\usepackage[all]{xy}
\usepackage{caption,subcaption}
\usepackage{pgfplots}
\usepackage{mathabx}

\pgfplotsset{soldot/.style={color=black,only marks,mark=*}} \pgfplotsset{holdot/.style={color=black,fill=white,only marks,mark=*}}

\newtheorem{thm}{Theorem}[section]

\newtheorem{prop}[thm]{Proposition}
\newtheorem{cor}[thm]{Corollary}
\newtheorem{lem}[thm]{Lemma}

\theoremstyle{definition}

\newtheorem{define}[thm]{Definition}

\theoremstyle{remark}
\newtheorem{remark}[thm]{Remark}
\newtheorem{remarks}[thm]{Remarks}

\newtheorem{example}[thm]{Example}
\newtheorem{examples}[thm]{Examples}

\numberwithin{equation}{section}

\newcommand{\an}{\CC^\omega}%{{\mathcal O}} 

 \newcommand{\J}{{\mathcal J}}

 \newcommand{\N}{{\mathbb N}}
 \newcommand{\R}{{\mathbb R}}
 \newcommand{\C}{{\mathbb C}}

\newcommand{\sph}{{\mathbb S}}

\newcommand{\Kk}{{\EuScript K}}

\newcommand{\Bb}{{\EuScript B}}

\newcommand{\Hh}{{\EuScript H}}

\newcommand{\tildebaja}{{\raise.17ex\hbox{$\scriptstyle\sim$}}}

\newcommand{\Int}{\operatorname{Int}}

\newcommand{\cl}{\operatorname{Cl}}

\newcommand{\dist}{\operatorname{dist}}
\newcommand{\id}{\operatorname{id}}

\newcommand{\Sing}{\operatorname{Sing}}

\newcommand{\x}{{\tt x}}

\newcommand{\ol}{\overline}

\newcommand{\veps}{\varepsilon}

\newcommand{\sfs}{\mathsf{s}}
\newcommand{\sft}{\mathsf{t}}
\newcommand{\sfu}{\mathsf{u}}

\newcommand{\mr}{\mathrm}
\newcommand{\mc}{\mathcal}
\newcommand{\mscr}{\mathscr}
\newcommand{\CC}{\mscr{C}}
\newcommand{\EE}{\mscr{E}}

\newcommand{\atsr}{\mscr{C}^r\text{-}\mathtt{ats}}
\newcommand{\atsinfty}{\mscr{C}^\infty\text{-}\mathtt{ats}}
\newcommand{\atsone}{\mscr{C}^1\text{-}\mathtt{ats}}

\newcommand{\atsrp}{\mscr{C}^r_*\text{-}\mathtt{ats}}
\newcommand{\atsinftyp}{\mscr{C}^\infty_*\text{-}\mathtt{ats}}

\begin{document}

\title[Smooth approximations in PL geometry]{Smooth approximations in PL geometry}

\author{Jos\'e F. Fernando}
\address{Departamento de \'Algebra, Geomatr\'\i a y Topolog\'\i a, Facultad de Ciencias Matem\'aticas, Universidad Complutense de Madrid, 28040 MADRID (SPAIN)}
\email{josefer@mat.ucm.es}
\thanks{The first author is supported by Spanish STRANO MTM2017-82105-P and Grupos UCM 910444. This article has been mainly written during a couple of one month research stays of the first author in the Dipartimento di Matematica of the Universit\`a di Trento. The first author would like to thank this department for the invitation and the very pleasant working conditions.}

\author{Riccardo Ghiloni}
\address{Dipartimento di Matematica, Via Sommarive, 14, Universit\`a di Trento, 38123 Povo (ITALY)}
\email{riccardo.ghiloni@unitn.it}
\thanks{The second author is supported by GNSAGA of INDAM} 

\begin{abstract}
Let $Y\subset\R^n$ be a triangulable set and let $r$ be either a positive integer or $r=\infty$. We say that $Y$ is a $\mscr{C}^r$-approximation target space, or a $\atsr$ for short, if it has the following universal approximation property: \em For each $m\in\N$ and each locally compact subset $X$ of~$\R^m$, each continuous map $f:X\to Y$ can be approximated by $\mscr{C}^r$ maps $g:X\to Y$ with respect to the strong Whitney $\mscr{C}^0$ topology\em. Taking advantage of new approximation techniques we prove: \em if $Y$ is weakly $\mscr{C}^r$ triangulable, then $Y$ is a $\atsr$\em. This result applies to relevant classes of triangulable sets, namely: (1) every locally compact polyhedron is a $\atsinfty$, (2) every set that is locally $\mscr{C}^r$ equivalent to a polyhedron is a $\atsr$ (this includes $\mscr{C}^r$ submanifolds with corners of $\R^n$) and (3) every locally compact locally definable set of an arbitrary o-minimal structure is a $\atsone$ (this includes locally compact locally semialgebraic sets and locally compact subanalytic sets). In addition, we prove: \em if $Y$ is a global analytic set, then each proper continuous map $f:X\to Y$ can be approximated by proper $\mscr{C}^\infty$ maps $g:X\to Y$\em. Explicit examples show the sharpness of our results.
\end{abstract}

\date{20/03/2021}

\maketitle 

%%%%%%%

\section{Introduction, main theorems and corollaries}\label{s1}

Approximation is a tool of great importance in many areas of mathematics. It allows to understand objects and morphisms of a certain category taking advantage of the corresponding properties of objects and morphisms in other categories that enjoy a better behavior and are dense inside the one we want to study.

In the geometrical context a remarkable example of an approximation result with thousand of applications concerns Whitney's approximation theorem \cite{wh2} for continuous maps whose target space is a $\mscr{C}^r$ submanifolds $Y$ of $\R^n$ for either a positive integer $r$ or $r=\infty$. An important fact is the existence of a system of $\mscr{C}^r$ tubular neighborhoods of $Y$ in $\R^n$ (together with the corresponding $\mscr{C}^r$ retractions onto $Y$). Whitney's approximation theorem can be used for instance to prove the existence of a unique $\mscr{C}^\infty$ manifold structure on each differentiable manifold of class $\mscr{C}^r$ for each positive integer $r$ (see \cite{hir3}).

This paper deals with the problem of approximating continuous maps by differentiable maps when the target space $Y\subset\R^n$ may have `singularities'. Actually, we require that $Y$ is at least triangulable.

The special case when the target space $Y\subset\R^n$ is a Nash set was already treated by Coste, Ruiz and Shiota in \cite{crs1}. In fact, they approximate real analytic maps on a compact Nash manifolds by a very restrictive class of approximating maps, the so-called Nash maps, see \cite[Ch.8]{bcr}. Recall that a function $f:\R^n\to\R$ is \em (real) Nash \em if it is of class $\mscr{C}^\omega$ (that is, real analytic) and algebraic over the real polynomials, that is, there exists a non-zero polynomial $P\in\R[x_1,\ldots,x_n,y]$ such that $P(x,f(x))=0$ for each $x\in\R^n$. In addition, $Y\subset\R^n$ is a \em Nash set \em if there exist a Nash function $f$ on $\R^n$ such that $Y=\{f=0\}$. A Nash set $X\subset\R^m$ that is in addition a smooth manifold is called a \em Nash manifold\em. A map $F:\R^m\to\R^n$ is {\it Nash} if its components are Nash functions. The \em Nash maps \em $f:X\to Y$ are the restrictions from $X$ to $Y$ of Nash maps $F:\R^m\to\R^n$ such that $F(X)\subset Y$. The authors proved in \cite[Thm 0.0]{crs1} a global version of Artin's approximation theorem \cite{ar}, which implies the following:

\begin{thm}[\cite{crs1}]\label{crs1}
Let $Y\subset\R^n$ be a Nash set and let $X\subset\R^m$ be any compact Nash manifold. Then every real analytic map $f:X\to Y$ can be uniformly approximated by (real) Nash maps $g:X\to Y$.
\end{thm}

The proof of the previous theorem is based on a deep result on commutative algebra: the so-called general N\'eron desingularization, see the survey \cite{crs2} for further reference. Lempert proved in \cite{le} the counterpart of Theorem \ref{crs1} for the complex setting taking advantage again of the general N\'eron desingularization: 

\begin{thm}[\cite{le}]
Let $Y\subset\C^n$ be a complex algebraic set and let $X$ be any holomorphically convex compact subset $X$ of a complex algebraic subset of $\C^m$. Then every holomorphic map $f:X\to Y$ can be uniformly approximated by complex Nash maps $g:X\to Y$.
\end{thm}
In \cite{bp} the authors provide a simpler proof of the previous statement based on strategies with a more geometric flavor.

Our main results in this work (Theorems \ref{thm:main1} and~\ref{thm:main2}) are of a different nature. We show: if $Y\subset\R^n$ belongs to a wide class of triangulable sets including differentiable manifolds, polyhedra, semi\-algebraic sets, subanalytic sets and definable sets of an o-minimal structure, then $Y$ enjoys the following approximation property as target space: {\it Let $X\subset\R^m$ be any locally compact set. Then each continuous map $f:X\to Y$ can be approximated by arbitrarily close $\mscr{C}^r$ maps $g:X\to Y$ for either suitable positive integers $r$ or $r=\infty$, with respect to the strong $\mscr{C}^0$ topology of $\mscr{C}^0(X,Y)$}. As the set $\mscr{C}_*^0(X,Y)$ of proper maps between $X$ and $Y$ is an open subset of $\mscr{C}^0(X,Y)$ (see \cite[Ch.2., Thm.1.5]{hir3}), the above $\mscr{C}^r$ approximation property implies the `proper $\mscr{C}^r$ approximation target space property'. We will revisit the latter property when dealing with $C$-analytic sets in \S\ref{C} below.

The preceding approximation by $\mscr{C}^r$ maps is always possible if $Y\subset\R^n$ is a $\mscr{C}^r$ submanifold with boundary (even if it has corners) or if $Y$ is a locally compact polyhedron. Our proofs introduce new approximation strategies that make use of a variant of the general simplicial approximation theorem, a `shrink-widen' covering and approximation technique and $\mscr{C}^r$ weak retractions. Our constructions provide also certain relative versions of the preceding approximation results, that is, results of the following form: {\em if in addition $X'\subset X$ is non-empty and $f|_{X'}$ is a $\mscr{C}^r$ map, there exist arbitrarily close $\mscr{C}^r$ maps $g:X\to Y$ to $f$ with respect to the strong $\mscr{C}^0$ topology of $\mscr{C}^0(X,Y)$ such that $g|_{X'}=f|_{X'}$}. Examples \ref{surp} show that we have to be quite restrictive with the hypotheses about $X'$ (even if $Y$ is as simple as a $\mscr{C}^r$ manifold with non-empty boundary).

In the literature there are many celebrated results concerning the existence of obstructions to approximate homeomorphisms between differentiable manifolds by diffeomorphisms. This obstruction theory is a central topic in differential topology, which was mainly developed by names like Milnor, Thom, Munkres and Hirsch in the fifties and sixties \cite{hir1,hir2,mi1,mu1,mu2,mu3,th}. We refer the reader also to \cite{DP,HM,IKO,MP,Mu} for some recent developments. Additional obstructions where found by Milnor in \cite{mi2} when he constructed two homeomorphic compact polyhedra which are not PL homeomorphic. Our results state that there are no obstructions to approximate continuous maps $f:X\to Y$ by differentiable maps $g:X\to Y$ when $Y$ admits a `good' triangulation. Of course, one cannot expect that the approximating map $g$ is a diffeomorphism or a PL homeomorphism if the map $f$ we want to approximate is a homeomorphism. In fact, the approximating maps $g$ we construct in this paper are far from being injective (see Remark \ref{rem:constant1}). 

%%%

\subsection{Weakly $\mscr{C}^r$ triangulable sets}

We assume in the whole article that every subset of $\R^n$ is endowed with the relative Euclidean topology (where $n\in\N:=\{0,1,2,\ldots\}$).

Let $Y\subset\R^n$ be a (non-empty) set. We say that $Y$ is {\it triangulable} if it is homeomorphic to a locally compact polyhedron of some $\R^q$. A {\it locally compact polyhedron} of $\R^q$ is defined as the realization $|L|$ of a locally finite simplicial complex $L$ of $\R^q$. For related notions concerning simplicial complexes we refer the reader to \cite{hu,mu5}.

Let $X\subset\R^m$ be a (non-empty) locally compact set and let $\mscr{C}^0(X,Y)$ be the set of all continuous maps from $X$ to $Y$. We endow $\mscr{C}^0(X,Y)$ with the {\it strong (Whitney) $\mscr{C}^0$ topology}. A fundamental system of neighborhoods of $f\in\mscr{C}^0(X,Y)$ in such a topology is given by the sets 
\[
\mc{N}(f,\varepsilon)=\big\{g\in\mscr{C}^0(X,Y):\ \|g(x)-f(x)\|_n<\varepsilon(x) \;\; \forall x\in X\big\},
\] 
where $\|\cdot\|_n$ denotes the Euclidean norm of $\R^n$ and $\varepsilon:X\to\R^+:=\{t\in\R:\ t>0\}$ is a strictly positive continuous function on $X$.

Denote $\N^*:=\N\setminus\{0\}$ the set of all positive integers and fix $r\in\N^*\cup\{\infty\}$. A map $g:X\to Y$ is a {\it $\mscr{C}^r$ map} if there exist an open neighborhood $U\subset\R^m$ of $X$ (in which $X$ is closed) and a differentiable map $G:U\to\R^n$ of class $\mscr{C}^r$ (in the standard sense) such that $g(x)=G(x)$ for each $x\in X$. Denote $\mscr{C}^r(X,Y)$ the subset of $\mscr{C}^0(X,Y)$ of all $\mscr{C}^r$ maps from $X$ to $Y$.

\begin{define}
Let $r\in\N^*\cup\{\infty\}$. A triangulable set $Y\subset\R^n$ is a \emph{$\mscr{C}^r$-approximation target space} or a $\atsr$ for short if $\mscr{C}^r(X,Y)$ is dense in $\mscr{C}^0(X,Y)$ for each locally compact subset $X$ of each Euclidean space $\R^m$, where $m\in\N$ is any natural number.
\end{define}

If $Y\subset\R^n$ is a $\atsr$, it is triangulable, so by \cite[Cor.3.5]{hanner} it is an absolute neighborhood retract. This implies that if $X\subset\R^m$ is an arbitrary locally compact set, $f\in\mscr{C}^0(X,Y)$ and $g\in\mscr{C}^r(X,Y)$ is any close enough approximation of $f$, then $g$ is homotopic to $f$ (see \cite[Thm.4.1]{hanner}). Thus, close enough approximations of $f$ are also homotopic between them. In fact, $\mscr{C}^r(X,Y)$ is not only dense in $\mscr{C}^0(X,Y)$ but it is also `homotopically dense' in $\mscr{C}^0(X,Y)$ in the following sense: \em for each $f\in\mscr{C}^0(X,Y)$ and each strictly positive continuous function $\varepsilon:X\to\R^+$ there exists $g\in\mc{N}(f,\varepsilon)\cap\mscr{C}^r(X,Y)$ that is homotopic to $f$\em. 

A natural question consists of determining if \em homotopic maps of $\mscr{C}^r(X,Y)$ \em and \em $\mscr{C}^r$ homotopic maps of $\mscr{C}^r(X,Y)$ \em coincide. \em In case $X$ is a locally compact set and $Y$ is a $\mscr{C}^r$ submanifold with boundary, then homotopic maps of $\mscr{C}^r(X,Y)$ are also $\mscr{C}^r$ homotopic maps of $\mscr{C}^r(X,Y)$ \em (see \cite[Ch.III., Thm.8.3]{orr}). The proof of this result uses the following $\mscr{C}^r$ approximation result \cite[Ch.III., Thm.6.1]{orr} relative to a closed subset $X'$ of $X$ (when $Y\subset\R^n$ is a $\mscr{C}^r$ submanifold with boundary).

\begin{thm}[{\cite[Ch.III., Thm.6.1]{orr}}]\label{thm:ORR}
Let $X\subset\R^m$ be a locally compact set, let $r\in\N^*\cup\{\infty\}$ and let $Y\subset\R^n$ be a $\mscr{C}^r$ submanifold with boundary. Let $X'\subset X$ be a closed set and let $f:X\to Y$ be a continuous map such that $f|_W$ is a $\mscr{C}^r$ map for some neighborhood $W\subset X$ of $X'$. Then there exists $g\in\mscr{C}^r(X,Y)$ arbitrarily close to $f$ in the strong $\mscr{C}^0$ topology such that $g$ coincides with $f$ in a neighborhood of $X'$ in $X$. 
\end{thm}

In this work we analyze when a triangulable set $Y\subset\R^n$ is a $\atsr$. Classical examples of $\atsr$ are the $\mscr{C}^r$ submanifolds of Euclidean spaces. Indeed, if $Y$ is any $\mscr{C}^r$ submanifold of~$\R^n$, then it is triangulable ($\mscr{C}^r$ triangulable indeed) by Cairns-Whitehead's triangulation theorem \cite{ca}. In addition, by Whitney's approximation theorem one can approximate each continuous map $f:X\to Y$ by an arbitrarily close $\mscr{C}^r$ map $g^*:X\to\R^n$ and then one can use a $\mscr{C}^r$ tubular neighborhood $\rho:U\to Y$ of $Y$ in $\R^n$ to define the $\mscr{C}^r$ map $g:X\to Y,\ x\mapsto g(x)=\rho(g^*(x))$ arbitrarily close to $f$ in the strong $\mscr{C}^0$ topology of $\mscr{C}^0(X,Y)$. 

%%%%%%%

If $Y\subset\R^n$ is an arbitrary triangulable set, a serious difficulty arises: $Y$ does not have $\mscr{C}^r$ tubular neighborhoods in $\R^n$ (recall that $r\geq 1$)! An easy counterexample is the $\mscr{C}^r$ triangulable set $Y:=\{xy=0\}\subset\R^2$. If there were a $\mscr{C}^1$ retraction $\rho:U\to Y$, then $d_0\rho$ would be the identity on $\R^2$ and the origin should be an interior point of $Y$ by the inverse function theorem, which is a contradiction. In fact, (boundaryless) $\mscr{C}^r$ submanifolds of Euclidean spaces can be characterized by the existence of $\mscr{C}^r$ retractions, namely: {\it if a subset of an Euclidean space can be covered by local $\mscr{C}^r$ retractions, it is a (boundaryless) $\mscr{C}^r$ submanifold} (see \cite[Thm.1.15]{Mic}).

In order to overcome this problem concerning the lack of $\mscr{C}^r$ tubular neighborhoods (when $Y$ is not a boundaryless $\mscr{C}^r$ submanifold of $\R^n$), we introduce the key concept of weakly $\mscr{C}^r$ triangulable set.

{\it Fix $r\in\N^*\cup\{\infty\}$ for a while.}

\begin{define}
We say that $Y\subset\R^n$ is \emph{weakly $\mscr{C}^r$ triangulable} if there exist a locally finite simplicial complex $L$ of some $\R^q$ and a homeomorphism $\Psi:|L|\to Y$ such that the restriction $\Psi|_\xi:\xi\to Y$ is a $\mscr{C}^r$ map for each simplex $\xi\in L$. 
\end{define}

Our first main result reads as follows. 

\begin{thm}\label{thm:main1}
Every weakly $\mscr{C}^r$ triangulable set is a $\atsr$.
\end{thm}

A natural matter that arises from the preceding statement is to reveal large and relevant fami\-lies of weakly $\mscr{C}^r$ triangulable sets. A first relevant example is the collection of $\mscr{C}^r$ submanifolds with boundary of $\R^n$ (also treated in \cite[Ch.III., Thm.6.1]{orr}):

\begin{cor}
Every $\mscr{C}^r$ submanifold with boundary of $\R^n$ is a $\atsr$.
\end{cor}

Another important example is given by the family of locally compact polyhedra itself. Indeed, each locally compact polyhedron is weakly $\mscr{C}^\infty$ triangulable by definition. Consequently:

\begin{cor}\label{cor:pl}
Every locally compact polyhedron is a $\atsinfty$. 
\end{cor}

Let $Y\subset\R^n$ be a triangulable set, let $L$ be a locally finite simplicial complex of some $\R^q$ and let $\Psi:|L|\to Y$ be a homeomorphism between the realization $|L|$ of $L$ and $Y$. Given $w\in L$ the \em star $\mr{St}(w,L)$ of $w$ in $L$ \em is the union of the interiors of those simplices of $L$ that have $w$ as a vertex \cite[\S2, p.11]{mu5}. Recall that the interior of a simplex $\sigma$ is defined as $\sigma$ with its proper faces removed \cite[\S1, p.5]{mu5}.

Recall that $Y\subset\R^n$ is said to be {\it $\mscr{C}^r$ triangulable} if there exists a homeomorphism $\Psi:|L|\to Y$ between the realization $|L|$ of a locally finite simplicial complex of some $\R^q$ and $Y$ such that: 
\begin{itemize}
\item the restriction $\Psi|_\xi:\xi\to Y$ is a $\mscr{C}^r$ map for each simplex $\xi$ of $L$ and 
\item the map $d\Psi_w:\mr{St}(w,L)\to\R^n,\ y\mapsto d(\Psi|_\xi)_w(y-w)$ (where $\xi$ is a simplex of $L$ such that $y\in\xi$) is a homeomorphism onto its image for each $w\in L$.
\end{itemize}
Every $\mscr{C}^r$ triangulable set is weakly $\mscr{C}^r$ triangulable by definition. For further reference concerning {\it $\mscr{C}^r$ triangulations} see for instance \cite{ca}, \cite[\S II.8]{mu4} and \cite[\S.I.3, pg. 72-94]{sh1}.

A set $Y\subset\R^n$ is called \em locally $\mscr{C}^r$ equivalent to a polyhedron\em, or {\em locally $\mscr{C}^r$~polyhedral} for short, if for each point $x\in Y$ there exist two open neighborhoods $U$ and $V$ of $x$ in $\R^n$, a $\mscr{C}^r$ diffeomorphism $\phi:U\to V$ and a locally compact polyhedron $P$ of $\R^n$ such that $x\in P$, $\phi(x)=x$ and $\phi(U\cap Y)=V\cap P$. In \cite[Prop.I.3.13]{sh1} Shiota proved that every locally $\mscr{C}^r$ polyhedral set is $\mscr{C}^r$ triangulable. We deduce:

\begin{cor}
Every $\mscr{C}^r$ triangulable set is a $\atsr$. In particular, every locally $\mscr{C}^r$~polyhedral set is a $\atsr$. 
\end{cor}

A celebrated family of locally $\mscr{C}^r$ polyhedral sets is the collection of $\mscr{C}^r$ submanifolds with corners of Euclidean spaces, which includes the above mentioned collection of $\mscr{C}^r$ submanifolds with boundary of Euclidean spaces. A subset $Y\subset\R^n$ is a {\it $\mscr{C}^r$ submanifold with corners} of dimension $d$ if for each point $x\in Y$ there exist an integer $k\in\{0,\ldots,d\}$ and open neighborhoods $U\subset\R^n$ of $x$ and $V\subset\R^n$ of the origin together with a $\mscr{C}^r$ diffeomorphism $\varphi:U\to V$ such that $\varphi(U\cap Y)=V\cap\{x_1\geq0,\ldots,x_k\geq0,x_{d+1}=0,\ldots,x_n=0\}\subset\R^n$, see \cite{jo,me1,me2}:

\begin{cor}
Every $\mscr{C}^r$ submanifold with corners of $\R^n$ is a $\atsr$.
\end{cor} 

Another well-known family of locally $\mscr{C}^r$ polyhedral sets arises when considering subsets $Y$ of a $\mscr{C}^r$ submanifold $M$ of dimension $d$ of some $\R^n$ with the following property: \em for each point $x\in Y$ there exists an open neighborhood $W\subset M$ of $x$ endowed with a $\mscr{C}^r$ diffeomorphism $\psi:W\to\R^p$ that maps $x$ to the origin and satisfies that $\psi(Y\cap W)$ is a union of coordinate vector subspaces of $\R^p$\em. Inside the preceding family appears unions of locally finite families of $\mscr{C}^r$ submanifolds of $M$ that meet transversally (in the preceding sense). Sets obtained in this way are called {\it sets with (only) $\mscr{C}^r$ monomial singularities} \cite{bfr,fgr}. A particular case concerns {\it $\mscr{C}^r$ normal-crossing divisors}, that is, unions of locally finite families of $\mscr{C}^r$ hypersurfaces of $M$ that meet transversally.

A very relevant class of triangulable sets is certainly the one of subanalytic sets, which includes semialgebraic sets. See \cite{bcr,bm,sh1} for basic facts concerning the geometry of these sets. Let us recall the main definitions.

A set $Y\subset\R^n$ is {\it semialgebraic} if it admits a description as a finite Boolean combination of polynomial equalities and inequalities. The set $Y$ is called {\it locally semialgebraic} if the intersection $Y \cap B$ is semialgebraic for each compact ball $B$ of $\R^n$.

Let $U\subset\R^n$ be a (non-empty) open set. A set $Y\subset U$ is {\it analytic} if for each point $x\in U$ there exists an open neighborhood $V\subset U$ of $x$ such that $Y\cap V=\{f_1=0,\ldots,f_r=0\}=\{f_1^2+\cdots+f_r^2=0\}$ for some real analytic functions $f_i\in\an(V,\R)$. More generally, a set $Y\subset U$ is {\it semianalytic} if for each point $x\in U$ there exists an open neighborhood $V\subset U$ of $x$ such that $Y\cap V$ is a finite Boolean combination of real analytic equalities and inequalities on $V$. The subanalytic sets are roughly speaking the images of semianalytic sets under proper real analytic maps. More precisely, $Y\subset U$ is a {\it subanalytic set} if there exist an open subset $W$ of some $\R^p$, a real analytic map $f:W\to U$ and a semianalytic set $T\subset W$ such that the restriction $f|_{\cl_W(T)}:\cl_W(T)\to U$ is proper and $f(T)=Y$. Here $\cl_W(T)$ is the closure of $T$ in $W$. Locally semialgebraic sets are semianalytic and hence subanalytic. 

The Hironaka-Hardt triangulation theorem \cite{hi1,hardt} asserts that each locally finite family ${\mathfrak Y}:=\{Y_i\}_{i\in I}$ of subanalytic subsets of $\R^n$ is `$\mscr{C}^\omega$ triangulable on open simplices' in the following sense \cite[Thms.1 \& 2]{hardt}: \em there exist a locally finite simplicial complex $L$ of some $\R^q$ and a homeomorphism $\Psi:|L|\to\R^n$ such that: 
\begin{itemize}
\item the image $\Psi(\xi^0)$ of each open simplex $\xi^0$ of $L$ is a $\mscr{C}^\omega$ submanifold of $\R^n$,
\item each restriction $\Psi|_{\xi^0}:\xi^0\to\Psi(\xi^0)$ is a $\mscr{C}^\omega$ diffeomorphism. 
\item each subanalytic set $Y_i$ is the (disjoint) union of some of the images $\Psi(\xi^0)$.
\end{itemize}\em
Unfortunately, this result does not ensure that a subanalytic subset $Y$ of $\R^n$ is weakly $\mscr{C}^r$ triangulable for some $r\in\N^*\cup\{\infty\}$.

The weakly $\mscr{C}^r$ triangulability of semialgebraic and subanalytic sets is not yet known for $r\geq2$. However, the situation is completely different for $r=1$. Indeed, in \cite{os} the authors have proved recently that every locally compact locally semialgebraic set $Y$ has a triangulating homeomorphism $\Psi:|L|\to Y$ such that $\Psi\in\mscr{C}^1(|L|,Y)$. In particular, $Y$ is weakly $\mscr{C}^1$ triangulable. See \cite{cp} for further information concerning the regularity of $\Psi$.

As it is commented by the authors of \cite{os} in the first paragraph of the introduction, it is straightforward to check that the techniques developed in \cite{os} extend to the subanalytic case. It turns out that locally compact subanalytic sets are weakly $\mscr{C}^1$ triangulable as well. We deduce:

\begin{cor} \label{cor:1ats}
Every locally compact subanalytic set is a $\atsone$. In particular, each locally compact locally semialgebraic set is a $\atsone$.
\end{cor}

Let us recall next the definition of o-minimal structure.
\begin{define} \label{def:o-minimal}
An \em o-minimal structure \em (on the field $\R$) is a collection ${\mathfrak S}:=\{{\mathfrak S}_n\}_{n\in\N^*}$ of
families of subsets of $\R^n$ satisfying:
\begin{itemize}
\item ${\mathfrak S}_n$ contains all the algebraic subsets of $\R^n$.
\item ${\mathfrak S}_n$ is a Boolean algebra.
\item If $A\in{\mathfrak S}_m$ and $B\in{\mathfrak S}_n$, then $A\times B\in{\mathfrak S}_{m+n}$.
\item If $\pi:\R^n\times\R\to\R^n$ is the natural projection and $A\in{\mathfrak S}_{n+1}$, then $\pi(A)\in{\mathfrak S}_n$.
\item ${\mathfrak S}_1$ consists precisely of all the finite unions of points and intervals of any type.
\end{itemize}
\end{define}
The elements of $\bigcup_{n\in\N^*}{\mathfrak S}_n$ are called \em definable sets of ${\mathfrak S}$\em. As a consequence of Tarski-Seidenberg theorem, semialgebraic sets constitute an o-minimal structure, which is in fact contained in each o-minimal structure. The collection of `global' subanalytic sets is precisely the collection of definable sets in the remarkable o-minimal structure $\R_{\rm an}$, see \cite{wi}. We refer the reader to \cite{vd2,vm} for further information on the celebrated theory of o-minimal structures. As in the semialgebraic case, we say that a set $Y\subset\R^n$ is a {\it locally definable set of ${\mathfrak S}$} if the intersection $Y \cap B$ is a definable set of $\mathfrak{S}$ for each compact ball $B$ of $\R^n$. 

Also in the o-minimal setting it is straightforward to adapt the constructions developed in \cite{os,cp} to show that every locally compact locally definable sets of any o-minimal structure is weakly $\mscr{C}^1$ triangulable (see the first paragraph of the introduction of \cite{os}). We deduce the following extension of Corollary \ref{cor:1ats}:

\begin{cor}
Every locally compact locally definable set of an arbitrary o-minimal structure is a $\atsone$.
\end{cor}

%%%

\subsection{$C$-analytic sets}\label{C}

Now, we focus on a quite significant subclass of subanalytic sets, the one of $C$-analytic sets (also known as global analytic sets \cite{c}). We do not know if $C$-analytic sets are weakly $\mscr{C}^r$ triangulable for some $r\geq2$, but we develop an alternative approximation strategy to prove in Theorem \ref{thm:main2} an analogous result to Theorem \ref{thm:main1} (with $r=\infty$) under the additional assumption that the involved maps are proper.

Let $U\subset\R^n$ be a (non-empty) open set. A set $Y\subset U$ is said to be a {\it $C$-analytic subset of $U$} if there exist finitely many global real analytic functions $f_1,\ldots,f_r\in\an(U,\R)$ such that 
\[
Y=\{f_1=0,\ldots,f_r=0\}=\{f_1^2+\cdots+f_r^2=0\}.
\]
By the term {\it $C$-analytic set} we mean a $C$-analytic subset of an open subset of some $\R^n$. Real algebraic sets and Nash sets are particular examples of $C$-analytic sets.

Let $X\subset\R^m$ be a locally compact set, let $Y\subset\R^n$ be a set and let $\mscr{C}^0_*(X,Y)$ be the set of all {\it proper} continuous maps from $X$ to $Y$ endowed with the relative topology inherited from the strong $\mscr{C}^0$ topology of $\mscr{C}^0(X,Y)$. Given $r\in\N^*\cup\{\infty\}$, we set $\mscr{C}^r_*(X,Y):=\mscr{C}^r(X,Y)\cap\mscr{C}^0_*(X,Y)$.

\begin{define}
Let $r\in\N^*\cup\{\infty\}$. A triangulable set $Y\subset\R^n$ is a \emph{$\mscr{C}^r_*$-approximation target space} or a $\atsrp$ for short if $\mscr{C}^r_*(X,Y)$ is dense in $\mscr{C}^0_*(X,Y)$ for each locally compact subset $X$ of each Euclidean space $\R^m$, where $m\in\N$ is any natural number.
\end{define}

Our second main result reads as follows.

\begin{thm} \label{thm:main2}
Every $C$-analytic set is a $\atsinftyp$.
\end{thm}

%%%

\subsection{Sharpness}

The results presented above provide families of triangulable sets $Y\subset\R^n$ for which $\mscr{C}^r(X,Y)$ is dense in $\mscr{C}^0(X,Y)$ where $r\in\N^*\cup\{\infty\}$ and $X\subset\R^m$ is an arbitrary locally compact set. If $s>0$ is any positive integer such that $s<r$, one may ask whether $\mscr{C}^r(X,Y)$ is also dense in $\mscr{C}^s(X,Y)$ at least in the case $X$ is a $\mscr{C}^s$ submanifold of $\R^m$, where $\mscr{C}^s(X,Y)$ is endowed with the relative topology induced by the strong (Whitney) $\mscr{C}^s$ topology of $\mscr{C}^s(X,\R^n)$ via the natural inclusion $\mscr{C}^s(X,Y)\hookrightarrow\mscr{C}^s(X,\R^n)$.

The following example points out that there is no hope to obtain general statements if $s>0$.

\begin{example}\label{c1top}
Let $X:=\{(x,y)\in\R^2:\ x^2+y^2=1\}$ be the standard circle, let $s\in\N^*$ and let $Y:=\{(x,y,z)\in X\times\R:\ z^3-y^{3s+1}=0\}$. Note that $Y$ is not a $\mscr{C}^{s+1}$ submanifold of $\R^3$.

Consider the $\mscr{C}^s$ map $f:X\to Y$, $(x,y)\mapsto(x,y,y^{s+1/3})$. Such a map cannot be $\mscr{C}^1$ approximated (and hence $\mscr{C}^s$ approximated) by maps in $\mscr{C}^{s+1}(X,Y)$. Suppose on the contrary that there exists a $\mscr{C}^{s+1}$ map $g:=(g_1,g_2,g_3):X\to Y$ arbitrarily close to $f$ in the strong $\mscr{C}^1$ topology. Thus, $g_*:=(g_1,g_2):X\to X$ is arbitrarily $\mscr{C}^1$ close to the identity map on $X$, hence $g_*$ is a $\mscr{C}^{s+1}$ diffeomorphism by the inverse function theorem. As $g_3=g_2^{s+1/3}$, it follows that $(g_3\circ g_*^{-1})(x,y)=y^{s+1/3}$ is a $\mscr{C}^{s+1}$ function on $X$, which is a contradiction. Indeed, such a function is not of class $\mscr{C}^{s+1}$ locally at $(\pm1,0)$. This proves that $\mscr{C}^{s+1}(X,Y)$ is not dense in $\mscr{C}^s(X,Y)$. $\sqbullet$
\end{example}

If $s=0$, it is also sharp our choice $r\in\N^*\cup\{\infty\}$, that is, we cannot choose in general $r=\omega$.

\begin{example}\label{sameexample}
Let $Y:=\{xy=0\}\subset\R^2$ and let $f:\R\to Y$ be the continuous map defined by $f(t):=(0,t)$ if $t<0$ and $f(t):=(t,0)$ if $t\geq 0$. Then $f$ cannot be approximated by real analytic maps $g=(g_-,g_+):\R\to Y$. Otherwise, $g_\pm$ would be a real analytic function on $\R$ vanishing identically locally at $\pm\infty$ and nowhere zero locally at $\mp\infty$, which is impossible by the principle of analytic continuation. The reader may compare this `negative' example with the `positive' approximation theorem \cite[Thm.1.7]{bfr}, which is a key result for the proof of the main theorem of \cite{fe}. $\sqbullet$
\end{example}

Similar examples appeared in our manuscript \cite{fg}.

%%%

\subsection{An unexpected by-product} \label{by}

The techniques involved in the proof of Theorem~\ref{thm:main1} reveal another approximation property of each locally compact polyhedron $P$ that has interest by its own. We fix a convention: \em The set $\mscr{C}^0(P,\R^+)$ of strictly positive continuous functions on $P$ is endowed with the partial ordering $\succcurlyeq$ defined by $\veps\succcurlyeq\delta$ if $\veps(w)\leq\delta(w)$ for each $w\in P$\em. Note that $\mscr{C}^0(P,\R^+)$ is a directed set with such an ordering.

\begin{cor} \label{cor:KP}
Let $K$ be a locally finite simplicial complex of $\R^p$ and let $P:=|K|\subset\R^p$ be its underlying locally compact polyhedron. Then there exists a net $\{\iota_\veps\}_{\veps\in\mscr{C}^0(P,\R^+)}$ in $\mscr{C}^\infty(P,P)$ that depends only on $K$, converges in the Moore-Smith sense to the identity map on $P$ in $\mscr{C}^0(P,P)$ and satisfies the following universal property:
\begin{itemize}
\item[$(\ast)$] Let $r\in\N^*\cup\{\infty\}$, let $Y\subset\R^n$ be any weakly $\mscr{C}^r$ triangulable set and let $f\in\mscr{C}^0(P,Y)$ be such that $f|_{\sigma}\in\mscr{C}^r(\sigma,Y)$ for each $\sigma\in K$. Then the net $\{f\circ\iota_\veps\}_{\veps\in\mscr{C}^0(P,\R^+)}$ converges in the Moore-Smith sense to $f$ in $\mscr{C}^0(P,Y)$ and each composition $f\circ\iota_\veps$ belongs to $\mscr{C}^r(P,Y)$.
\end{itemize}
\end{cor}
\begin{remark}\label{rem:KP}
If in the preceding statement $K$ is a finite simplicial complex, the net $\{\iota_\veps\}_{\veps\in\mscr{C}^0(P,\R^+)}$ can be replaced by a sequence $\{\iota_k\}_{k\in\N}$ in $\mscr{C}^\infty(P,P)$ with the same universal property. $\sqbullet$
\end{remark}

%%%

\subsection{Smooth relative approximations}\label{ra}
A natural question that arises when dealing with approximation problems concerns the existence of relative versions.

{\it Fix $r\in\N^*\cup\{\infty\}$.}

Let $X'\subset X\subset\R^m$ be non-empty sets such that $X$ is locally compact and $X'$ is closed in $X$, let $Y\subset\R^n$ be a set and let $f:X\to Y$ be a continuous map whose restriction $f|_{X'}:X'\to Y$ to $X'$ is a $\mscr{C}^r$ map. \em Are there $\mscr{C}^r$ maps $g:X\to Y$ that approximate $f$ and satisfy $g|_{X'}=f|_{X'}$?\em

As it is well-known \cite[p.42, Ex.(a)]{mu4}: \em If $Y$ admits a system of $\mscr{C}^r$ tubular neighborhoods in $\R^n$ \em (that is, $Y$ is a boundaryless $\mscr{C}^r$ submanifold of $\R^n$)\em, then relative approximation is always possible\em. We sketch here a proof of the previous fact to stress once more the crucial role played by $\mscr{C}^r$ retractions onto $Y$.
\begin{proof}[Sketch of proof]
Let $f\in\mscr{C}^0(X,Y)$ be such that $f|_{X'}$ is a $\mscr{C}^r$ map. Let $(V,\rho)$ be a $\mscr{C}^r$ tubular neighborhood of $Y$ in $\R^n$, where $\rho:V\to Y$ is a $\mscr{C}^r$ retraction. By Whitney's approximation theorem there exists $g_0\in\mscr{C}^r(X,Y)$ close to $f$ in the strong $\mscr{C}^0$ topology. Let $f_1:X\to\R^n$ be a $\mscr{C}^r$ extension of $f|_{X'}$ to a small enough open neighborhood $U\subset X$ of $X'$. Let $\{\theta,1-\theta\}$ be a $\mscr{C}^r$ partition of unity associated to $\{U,X\setminus X'\}$. Then $g_1:=\theta f_1+(1-\theta)g_0:X\to\R^n$ is close to $f$ in the strong $\mscr{C}^0$ topology and $g_1|_{X'}=f|_{X'}$. We may assume in addition $g_1(X)\subset V$ and $g:=\rho\circ g_1$ is close to $\rho\circ f=f$ in the strong $\mscr{C}^0$ topology (see Lemma \ref{lem:approx} below). Observe that $g|_{X'}=g_1|_{X'}=f|_{X'}$, as required. 
\end{proof}

For more general spaces $Y$, which do not have $\mscr{C}^r$ retractions of open neighborhoods of $Y$ in $\R^n$ onto $Y$, the situation is more restrictive with both $X'$ and the restriction $f|_{X'}$. Let us see some enlightening counterexamples concerning $\mscr{C}^r$ manifolds $Y$ with non-empty boundary.

\begin{examples}\label{surp}
Let $Y$ be a $\mscr{C}^r$ submanifold of $\R^n$ with non-empty boundary. Let $D(Y)$ be the double of $Y$. Denote $T_yY$ the tangent space of $Y$ at the point $y\in Y$. Let $h:Y\to\R$ be a $\mscr{C}^r$ equation of the boundary $\partial Y$ of $Y$ such that $\{h>0\}$ equals the interior $Y\setminus\partial Y$ of $Y$ and $d_yh:T_yY\to\R$ is surjective for each $y\in\partial Y$. A well-know way to endow $D(Y)$ with a $\mscr{C}^r$ structure is to identify it with the boundaryless $\mscr{C}^r$ manifold $M:=\{(x,t)\in Y\times\R:\ t^2=h(x)\}$. We also identify $Y$ with $M\cap\{t\geq0\}$. Denote $\pi:\R^n\times\R\to\R$ the projection $(x,t)\mapsto t$, which satisfies $\pi(Y)\subset[0,+\infty)$ and $\pi(\partial Y)=\{0\}$.

(i) Let $X'\subset X\subset\R^m$ be non-empty sets and let $f:X\to Y$ be a continuous map such that the restriction $f|_{X'}:X'\to Y$ is a $\mscr{C}^r$ map. Assume that there exists a continuous path $\beta:[-1,1]\to X$ such that $\beta([0,1])\subset X'$, the restriction $\beta|_{[0,1]}:[0,1]\to X$ is a $\mscr{C}^r$ map, $f(\beta(0))\in\partial Y$ and the derivative $(\pi\circ f\circ\beta)'(0)$ of $\pi\circ f\circ\beta$ at $0$ is strictly positive. Recall that $\pi(Y)\subset[0,+\infty)$ and $\pi(\partial Y)=\{0\}$. Consider the continuous map 
\[
f^*:[-1,1]\to[0,+\infty),\ t\mapsto\pi(f(\beta(t))). 
\]
Note that $f^*|_{[0,1]}$ is a $\mscr{C}^r$ map, $f^*(0)=0$ and $(f^*)'(0)>0$. It follows immediately that there exists no $\mscr{C}^r$ map $g:X\to Y$ such that $g|_{X'}=f|_{X'}$ (or better $g|_{\beta([0,1])}=f|_{\beta([0,1])}$). Otherwise, the $\mscr{C}^r$ map $g^*:[-1,1]\to[0,+\infty)$, $t\mapsto\pi(g(\beta(t)))$ would coincide with $f^*$ on $[0,1]$, so $(g^*)'(0)=(f^*)'(0)>0$, so $g^*$ would be strictly increasing locally at $t=0$ in $[-1,1]$. This is impossible because $g^*(0)=f^*(0)=0$ and $g^*\geq0$ on the whole interval $[-1,1]$. Consequently, there exists no $\mscr{C}^r$ extension of $f^*|_{[0,1]}$ to $[-1,1]$ whose image is contained in $[0,+\infty)$ and there exists no $\mscr{C}^r$ extension from $X$ to $Y$ of $f|_{X'}$. In particular, there exists no $\mscr{C}^r$ approximation $g$ of $f$ such that $g|_{X'}=f|_{X'}$.

$(\mr{i}')$ An easy example in which the obstruction described in (i) appears is the following: $X:=[-1,1]$, $X':=[0,1]$, $Y:=\{(x,y)\in\R^2:\ y\geq0\}$ and $f:X\to Y$ is given by $f(t):=(t,0)$ if $t\in[-1,0)$, $f(t):=(0,t)$ if $t\in X'=[0,1]$ and $\beta:[-1,1]\to X$ is the identity map.

(ii) We keep the notations fixed above concerning the $\mscr{C}^r$ submanifold $Y$ of $\R^n$ with non-empty boundary. Suppose that $Y$ has dimension $d$. Consider the continuous map
\[
f_1:D(Y)\to Y,\ (x,t)\mapsto(x,|t|).
\]
We claim: \em $f_1|_Y=\id_Y$ (which is $\mscr{C}^r$), but $f_1|_Y$ admits no $\mscr{C}^r$ extension to $D(Y)$ whose image is contained in $Y$\em. This implies that \em there exists no $\mscr{C}^r$ approximation $g$ of $f_1$ such that $g|_Y=f_1|_Y$\em.

Pick a point $y_0\in\partial Y$. Let $\epsilon>0$, let $B_n(0,\epsilon^2)$ be the open ball of $\R^n$ centered at $0$ with radius $\epsilon^2$ and let $U\subset\R^n$ be an open neighborhood of $y_0$ endowed with a $\mscr{C}^r$ diffeomorphism $u:=(u_1,\ldots,u_n):U\to B_n(0,\epsilon^2)$ such that $u(y_0)=0$, $u(U\cap Y)=B_n(0,\epsilon^2)\cap\{x_1\geq0,x_{d+1}=0,\ldots,x_n=0\}$ and $u_1(x)=h(x)$ for each $x\in U\cap Y$. Consider the $\mscr{C}^r$ diffeomorphism
\[
u^*:U^*:=U\times\R\to B_n(0,\epsilon^2)\times\R,\ (x,t)\mapsto(u(x),t)
\]
such that $u^*(y_0,0)=(0,0)$ and 
\[
u^*(D(Y)\cap U^*)=(B_n(0,\epsilon^2)\times\R)\cap\{x_1=t^2,x_{d+1}=0,\ldots,x_n=0\}=:M^*.
\] 
Consider the path $\alpha:(-\epsilon,\epsilon)\to M^*$, $t\mapsto(t^2,0,\ldots,0,t)$ and define the function
\[
f_1^*:(-\epsilon,\epsilon)\to\R,\ t\mapsto\pi(f_1((u^*)^{-1}(\alpha(t))))=|t|.
\] 
Observe that $f_1^*|_{[0,\epsilon)}:[0,\epsilon)\to\R$, $t\mapsto t$ is a $\mscr{C}^r$ function, but there exists no extension of $f_1^*|_{[0,\epsilon)}$ to $(-\epsilon,\epsilon)$ whose image is contained in $\pi(Y)\subset[0,+\infty)$. In particular, there exists no $\mscr{C}^r$ extension of $f_1|_Y$ to $D(Y)$ whose image is contained in $Y$, as required. $\sqbullet$
\end{examples}

In order to avoid the obstruction described in the preceding examples, in \cite[Ch.III., Thm.6.1]{orr} (see Theorem \ref{thm:ORR} above) the authors slightly modify classical relative approximation statement and ask (as a stronger hypothesis) that the restriction of the continuous function $f$ to a small enough open neighborhood $W\subset X$ of $X'$ is a $\mscr{C}^r$ map. Then they take advantage of the existence of $\mscr{C}^r$ tubular neighborhoods of $Y\setminus\partial Y$ in $\R^n$ and of $\mscr{C}^r$ collars of $\partial Y$ in $Y$. In this way, they obtain a $\mscr{C}^r$ approximation $g$ of $f$ which coincides with $f$ on a neighborhood $W'\subset W$ of $X'$. Since $g=f$ on a neighborhood $W'\subset W$ of $X'$, it seems not possible to integrate the constructions in \cite[Ch.III., Thm.6.1]{orr} with our approximation method to achieve more general situations, unless $f$ is constant on $W$ (see Remark \ref{rem:constant1} below). 

Suppose next $X'$ is a discrete and closed subset of $X$. Our next result states that $\mscr{C}^r$ approxi\-mation of continuous maps $f:X\to Y$ relative to $X'$ is possible.

\begin{thm}\label{thm:relative-discrete}
Let $X\subset\R^m$ be a locally compact set, let $X'$ be a discrete and closed subset of $X$, let $Y\subset\R^n$ be a weakly $\mscr{C}^r$ triangulable set and let $f:X\to Y$ be a continuous map. If $f(X')$ is discrete and closed in $Y$, then there exists $g\in\mscr{C}^r(X,Y)$ arbitrarily close to $f$ in the strong $\mscr{C}^0$ topology such that $g|_{X'}=f|_{X'}$.
\end{thm}

The latter result implies straightforwardly some properties of weakly $\mscr{C}^r$ triangulable sets $Y$ concerning connectedness and homotopy. Fix $y_0\in Y$ and denote $\pi_p(Y,y_0)$ the $p^{\mr{th}}$ homotopy group of the pointed space $(Y,y_0)$ for each $p\in\N^*$. We understand the elements of $\pi_p(Y,y_0)$ as the homotopy classes of continuous maps $(\sph^p,N)\to(Y,y_0)$, where $\sph^p$ is the standard $p$-sphere and $N$ is its north pole. The path-connected components of $Y$ coincide with its $\mscr{C}^r$ path-connected components and each element of $\pi_p(Y,y_0)$ has a representative of class $\mscr{C}^r$.

\begin{cor}\label{cor:topol}
Let $Y\subset\R^n$ be a weakly $\mscr{C}^r$ triangulable set. We have:
\begin{itemize}
\item[(i)] Each continuous path $\gamma:[0,1]\to Y$ can be approximated in the strong $\mscr{C}^0$ topology by $\mscr{C}^r$ paths $\alpha:[0,1]\to Y$ such that $\alpha(0)=\gamma(0)$ and $\alpha(1)=\gamma(1)$.
\item[(ii)] Every element of $\pi_p(Y,y_0)$ can be represented by a $\mscr{C}^r$ map.
\end{itemize}
\end{cor}

The preceding corollary holds for $r=\infty$ if $Y$ is a locally compact polyhedron and for $r=1$ if $Y$ is a locally compact locally definable set in an arbitrary o-minimal structure.

Theorem \ref{thm:relative-discrete} can still be extended: the crucial property to have approximations relative to $X'$ is that $f(X')$ has no accumulation points in $Y$. Before entering into details, we introduce the following definition.

\begin{define}\label{def:w*-C^r-triangulation-p}
Let $X'\subset X\subset\R^m$ be sets such that ($X$ is non-empty and) $X'$ is closed in $X$. The pair $(X,X')$ is \emph{weakly$^*$ $\mscr{C}^r$ triangulable} if there exist a locally finite simplicial complex $K$ of some $\R^p$, a subcomplex $K'$ of $K$ and a homeomorphism $\Phi:|K|\to X$ such that $\Phi(|K'|)=X'$ and the restriction $\Phi|_{\sigma^0}:\sigma^0\to X$ is a $\mscr{C}^r$ map for each open simplex $\sigma^0$ of $K$. We say that $X$ is \emph{weakly$^*$ $\mscr{C}^r$ triangulable} if so is the pair $(X,\emptyset)$.
\end{define}

Let $X'\subset X\subset\R^m$ be such that $X$ is locally compact and $X'$ is closed in $X$. The pair $(X,X')$ is a \em subanalytic pair \em if both $X$ and $X'$ are subanalytic subsets of $\R^m$. Analogously, if ${\mathfrak S}$ is an o-minimal structure (on the field $\R$), the pair $(X,X')$ is a \em locally definable pair of ${\mathfrak S}$ \em if both $X$ and $X'$ are locally definable sets of ${\mathfrak S}$. By Hironaka-Hardt's triangulation theorem a subanalytic pair is a weakly$^*$ $\mscr{C}^\infty$ triangulable pair, whereas by \cite[Ch.II.Thm.II']{sh1} a locally definable pair is a weakly$^*$ $\mscr{C}^\infty$ triangulable pair.

The announced extension of Theorem \ref{thm:relative-discrete} is the following.

\begin{thm} \label{thm:relative-locally-constant}
Let $(X,X')$ be a weakly$^*$ $\mscr{C}^r$ triangulable pair, let $Y\subset\R^n$ be a weakly $\mscr{C}^r$ triangulable set and let $f:X\to Y$ be a continuous map such that $f(X')$ is discrete and closed in $Y$. Then, there exists $g\in\mscr{C}^r(X,Y)$ arbitrarily close to $f$ in the strong $\mscr{C}^0$ topology such that $g|_{X'}=f|_{X'}$.
\end{thm}

As an immediate consequence of the preceding result, one deduces the following:

\begin{cor}
Let $(X,X')$ be either a subanalytic pair or a locally definable pair of an arbitrary o-minimal structure $\mathfrak{S}$, let $Y\subset\R^n$ be a set and let $f:X\to Y$ be a continuous map such that $f(X')$ is discrete and closed in $Y$. We have:
\begin{itemize}
 \item[(i)] If $(X,X')$ is a subanalytic pair and $Y$ is a locally compact polyhedron, there exists $g\in\mscr{C}^\infty(X,Y)$ arbitrarily close to $f$ in the strong $\mscr{C}^0$ topology such that $g|_{X'}=f|_{X'}$.
 \item[(ii)] If $(X,X')$ is a locally definable pair of $\mathfrak{S}$ and $Y$ is a locally compact locally definable set of $\mathfrak{S}$, there exists $g\in\mscr{C}^1(X,Y)$ arbitrarily close to $f$ in the strong $\mscr{C}^0$ topology such that $g|_{X'}=f|_{X'}$.
\end{itemize} 
\end{cor}

\subsection{Structure of the article}
In Section \ref{s2} we collect a couple of preliminary results concerning spaces of continuous maps. In the first part of Section \ref{s3} we present our variant of the general approximation theorem and our `shrink-widen' covering and approximation technique. In the second part we combine these results with the ones in Section \ref{s2} to prove Theorem \ref{thm:main1}. We provide also the proof of Corollary~\ref{cor:KP} and the relative approximation Theorem \ref{thm:relative-locally-constant} (from which follow readily the other results in \S\ref{ra}). Section \ref{s4} is devoted to prove Theorem \ref{thm:main2}, which involves the proof of the existence of \em $\mscr{C}^r$ weak retractions \em and the immersion of $C$-analytic sets as singular sets of coherent $C$-analytic sets homeomorphic to Euclidean spaces. Weaker and purely semialgebraic versions of some results presented in this article appeared in our preceding manuscript \cite{fg}.

\subsection*{Acknowledgements.}
The authors are indebted with the anonymous referee for very valuable suggestions to improve the presentation of this article.

%%%%%%%

\section{Preliminaries on spaces of continuous maps}\label{s2}

In this short section we collect a couple of results useful for the sequel. First we fix two notations we will use freely throughout the manuscript. Let $S,T\subset\R^q$ be such that $S\subset T$. Denote respectively $\cl_T(S)$ and $\Int_T(S)$ the closure of $S$ in $T$ and the interior of $S$ in $T$. The following result is well-known and its proof follows straightforwardly from \cite[\S2.5. Ex.10, pp. 64-65]{hir3} using standard arguments.

\begin{lem}\label{lem:approx}
Let $X\subset\R^m$, $X'\subset\R^{m'}$, $Y\subset\R^n$ and $Y'\subset\R^{n'}$ be locally compact sets, let $f:Y\to Y'$ be an arbitrary continuous map and let $g:X\to X'$ be a proper continuous map. Then the maps
\[
f_*:\mscr{C}^0(X,Y)\to\mscr{C}^0(X,Y'),\ h\mapsto f\circ h
\quad\text{and}\quad
g^*:\mscr{C}^0(X',Y)\to\mscr{C}^0(X,Y),\ h\mapsto h\circ g
\]
are continuous.
\end{lem}

As a consequence, we deduce:

\begin{cor}\label{cor:closedness}
Let $X\subset\R^m$ and $Y\subset\R^n$ be locally compact sets. Then there exist closed subsets $X'$ of $\R^{m+1}$ and $Y'$ of $\R^{n+1}$ such that: 
\begin{itemize}
\item $X'$ is homeomorphic to $X$ and $Y'$ is homeomorphic to $Y$, 
\end{itemize}
and the following property holds for each $r\in\N^*\cup\{\infty\}$: 
\begin{itemize}
\item $\mscr{C}^r(X,Y)$ is dense in $\mscr{C}^0(X,Y)$ if and only if $\mscr{C}^r(X',Y')$ is dense in $\mscr{C}^0(X',Y')$.
\end{itemize}
In addition, if $Y$ is a $C$-analytic subset of some open subset $U$ of $\R^n$, then there exists also a $C$-analytic subset $Y''$ of $\R^{2n+1}$ homeomorphic to $Y$ such that $\mscr{C}^\infty_*(X,Y)$ is dense in $\mscr{C}^0_*(X,Y)$ if and only if $\mscr{C}^\infty_*(X',Y'')$ is dense in $\mscr{C}^0_*(X',Y'')$.
\end{cor}
\begin{proof}
As $X\subset\R^m$ and $Y\subset\R^n$ are locally compact (or equivalently locally closed) sets, the differences $\cl_{\R^m}(X)\setminus X$ and $\cl_{\R^n}(Y)\setminus Y$ are respectively closed in $\R^m$ and in $\R^n$. Let $\theta:\R^m\to\R$ and $\xi:\R^n\to\R$ be $\mscr{C}^\infty$ functions such that $\theta^{-1}(0)=\cl_{\R^m}(X)\setminus X$ and $\xi^{-1}(0)=\cl_{\R^n}(Y)\setminus Y$. Define $X':=\{(x,t)\in X\times\R:\ t=1/\theta(x)\}$ and $Y':=\{(y,t)\in Y\times\R:\ t=1/\xi(y)\}$ and consider the homeomorphisms $\Theta:X'\to X$ and $\Xi:Y\to Y'$ given by $\Theta(x,t):=x$ and $\Xi(y):=(y,1/\xi(y))$. The sets $X'$ and $Y'$ are respectively closed in $\R^{m+1}$ and in $\R^{n+1}$. By Lemma \ref{lem:approx} the map $H:\mscr{C}^0(X,Y)\to\mscr{C}^0(X',Y')$ given by $H:=\Theta^*\circ\Xi_*$ is a homeomorphism. As $\Theta$, $\Theta^{-1}$, $\Xi$ and $\Xi^{-1}$ are $\mscr{C}^\infty$ maps, we deduce $H(\mscr{C}^r(X,Y))=\mscr{C}^r(X',Y')$ for each $r\in\N^*\cup\{\infty\}$, so the first part of the statement is proved.

Let us prove the second part. By Whitney's embedding theorem for the real analytic case \cite[2.15.12]{n}, there exists a real analytic embedding $\varphi:U\to\R^{2n+1}$ such that $M:=\varphi(U)$ is a closed real analytic submanifold of $\R^{2n+1}$. By Cartan's Theorem B real analytic functions on $M$ are restrictions to $M$ of real analytic functions on $\R^{2n+1}$. Thus, $Y'':=\varphi(Y)$ is a $C$-analytic subset of $\R^{2n+1}$. Denote $\Phi:Y\to Y''$ the restriction of $\varphi$ from $Y$ to $Y''$ and $H':\mscr{C}^0(X,Y)\to\mscr{C}^0(X',Y'')$ the homeomorphism $H':=\Theta^*\circ\Phi_*$. We conclude $H'(\mscr{C}^0_*(X,Y))=\mscr{C}^0_*(X',Y'')$ and $H'(\mscr{C}^\infty_*(X,Y))=\mscr{C}^\infty_*(X',Y'')$, as required. 
\end{proof}

\begin{remark} \label{rem:topologies}
Let $X\subset\R^m$ and let $Y\subset\R^n$ be (non-empty) sets such that $X$ is locally compact. Consider a locally finite covering $\{C_\ell\}_{\ell\in L}$ of $X$ by non-empty compact sets and a family $\{\veps_\ell\}_{\ell\in L}$ of positive real numbers. Making use of a suitable $\mscr{C}^0$ partition of unity on $X$, one shows the existence of a strictly positive continuous function $\veps:X\to\R^+$ such that $\max_{C_\ell}(\veps)\leq\veps_\ell$ for each $\ell\in L$. This implies that a fundamental system of neighborhoods of $f\in\mscr{C}^0(X,Y)$ for the strong $\mscr{C}^0$ topology of $\mscr{C}^0(X,Y)$ is given by the sets
\[
\mc{N}(f,\{C_\ell\}_{\ell\in L},\{\veps_\ell\}_{\ell\in L}):=\{g\in\mscr{C}^0(X,Y) : \|g(x)-f(x)\|_n<\veps_\ell \;\; \forall\ell\in L,\, \forall x\in C_\ell\},
\]
where $\{C_\ell\}_{\ell\in L}$ runs over the locally finite coverings of $X$ by non-empty compact sets and $\{\veps_\ell\}_{\ell\in L}$ runs over the families of positive real numbers with the same set $L$ of indices. $\sqbullet$
\end{remark}

%%%%%%%

\section{Proofs of Theorem \ref{thm:main1} and Corollary \ref{cor:KP}}\label{s3}

In this section we develop first all the machinery we need to prove Theorem \ref{thm:main1}: 
\begin{itemize}
\item a variant of the general simplicial approximation theorem (that appears in \S\ref{zee}),
\item a `shrink-widen' covering and approximation technique (that appears in \S\ref{swt}),
\end{itemize}
and after we approach its proof (see \S\ref{thm11}). Finally, we prove Corollary \ref{cor:KP} and Theorems \ref{thm:relative-discrete} and \ref{thm:relative-locally-constant} (see \S\ref{subsec:rlc}, \S\ref{subsec:rd} and \S\ref{co14}) in the required order. A weaker `finite' version of the `shrink-widen' covering and approximation technique, that we present here in \S\ref{swt}, is contained in our manuscript \cite{fg}. 
%%%

\subsection{A variant of the general simplicial approximation theorem}\label{zee}

Given a locally finite simplicial complex $K$ of some $\R^p$, a \em subdivision $K'$ of $K$ \em is a locally finite simplicial complex $K'$ of $\R^p$ such that $|K'|=|K|$ and each simplex of $K'$ is a subset of some simplex of $K$. A particular case of subdivision of $K$ is the first barycentric subdivision $\mr{sd}(K)$ of $K$. We denote $\mr{sd}^k(K):=\mr{sd}(\mr{sd}^{k-1}(K))$ the \em $k^{\mathit{th}}$ barycentric subdivision of $K$ \em for $k\geq1$, where $\mr{sd}^0(K):=K$.

Let $L$ be a locally finite simplicial complex of some $\R^q$ and let $F:|K|\to|L|$ be a continuous map. A simplicial map $F^\bullet:|K|\to|L|$ is said to be a {\it simplicial approximation of $F$} if $F(\mr{St}(v,K))\subset\mr{St}(F^\bullet(v),L)$ for each vertex $v$ of $K$. If $w\in|L|$, the \em carrier of $w$ in $L$ \em is the unique simplex $\tau\in L$ such that $w\in\tau^0$. The classical simplicial approximation theorem asserts:

\begin{thm}[\cite{alexander}]\label{thm:c-s-a-th}
If $K$ and $L$ are finite simplicial complexes, there exists a natural number $k$ such that $F$ has a simplicial approximation $F^\bullet:|\mr{sd}^k(K)|\to|L|$. In addition, if $w\in|K|$ and $\tau$ is the carrier of $F(w)$ in $L$, then $F^\bullet(w)\in\tau$.
\end{thm}

As an immediate consequence: 

\begin{cor}\label{cor:s-approx}
If $K$ and $L$ are finite simplicial complexes and $\varepsilon>0$ is a positive real number, then there exist two natural numbers $\kappa$ and $\ell$ and a simplicial map $F^*:|\mr{sd}^{\kappa}(K)|\to|\mr{sd}^\ell(L)|$ such that $\|F^*(w)-F(w)\|_q<\varepsilon$ for each $w\in|\mr{sd}^{\kappa}(K)|=|K|$.
\end{cor}

We need to extend the latter result for locally finite simplicial complexes $K$ and $L$ with respect to the strong $\mscr{C}^0$ topology of $\mscr{C}^0(|K|,|L|)$.

Let $K$ and $L$ be arbitrary locally finite simplicial complexes and let $F:|K|\to|L|$ be a continuous map. We say that $F$ satisfies the {\it star condition (relative to $K$ and $L$)} if for each vertex $v$ of $K$ there exists a vertex $w$ of $L$ such that $F(\mr{St}(v,K))\subset\mr{St}(w,L)$.

\begin{lem}[{\cite[Lem.14.1(a)(b)]{mu5}}] \label{lem:star-condition}
If the continuous map $F:|K|\to|L|$ satisfies the star condition, then it has a simplicial approximation $F^\bullet:|K|\to|L|$.
\end{lem}

Given two coverings $\mc{A}$ and $\mc{B}$ of $|K|$, we say that {\it $\mc{B}$ refines $\mc{A}$} if for each $B\in\mc{B}$, there exists $A\in\mc{A}$ such that $B\subset A$. If $v$ is a vertex of $K$, the {\it closed star $\ol{\mr{St}}(v,K)$ of $v$ in $K$} is defined as the closure of $\mr{St}(v,K)$ in $|K|$. Observe that $\ol{\mr{St}}(v,K)$ is the union of all simplices of $K$ having $v$ as a vertex. In particular, it is the realization of a simplicial subcomplex of $K$.

\begin{thm}[{\cite[Thm.16.4]{mu5}}]\label{thm:refine}
Let $K$ be a locally finite simplicial complex and let $\mc{A}$ be an open covering of $|K|$. Then there exists a subdivision $K'$ of $K$ such that the collection of closed stars $\overline{\mr{St}}(v,K')$, where $v$ ranges over the vertices of $K'$, refines $\mc{A}$. 
\end{thm}

If we define $\mc{A}$ as the collection of $F^{-1}(\mr{St}(w,L))$, where $w$ ranges over the vertices of $L$, then there exists by Theorem \ref{thm:refine} a subdivision $K'$ of $K$ whose closed stars refines $\mc{A}$. Consequently, $F$ satisfies the star condition relative to $K'$ and $L$ and Lemma \ref{lem:star-condition} implies:

\begin{thm}[General simplicial approximation {\cite[Thm.16.5]{mu5}}] \label{thm:g-s-a-th}
Given a continuous map $F:|K|\to|L|$ between locally compact polyhedra, there exists a subdivision $K'$ of $K$ such that $F$ has a simplicial approximation $F^\bullet:|K'|\to|L|$.
\end{thm}

As pointed out above, we need a suitable version of the preceding theorem that takes into account, not only simplicial approximation, but also strong $\mscr{C}^0$ approximation: as Corollary \ref{cor:s-approx} does with respect to the classical simplicial approximation Theorem \ref{thm:c-s-a-th}. To do this, we introduce the notion of weakly simplicial map. 

\begin{define}
Let $K$ and $L$ be locally finite simplicial complexes and let $F:|K|\to|L|$ be a continuous map. Suppose $|K|\subset\R^p$ and $|L|\subset\R^q$. We say that $F$ is \emph{weakly simplicial} if, for each simplex $\sigma\in K$, there exist a simplex $\xi_\sigma\in L$ and an affine map $A_\sigma:\R^p\to\R^q$ such that $F(\sigma)\subset\xi_\sigma$ and $F(x)=A_\sigma(x)$ for each $x\in\sigma$.
\end{define}

Observe that each weakly simplicial map $F:|K|\to|L|$ is uniquely determined by their values on the vertices of $K$. Evidently, each simplicial map is weakly simplicial. On the contrary, the map $F:\{0\}\to[-1,1]$, $0\mapsto0$ is an easy example of a weakly simplicial map between polyhedra of $\R$ which is not simplicial, if $[-1,1]$ is the realization of the simplicial complex $K:=\{\{-1\},\{1\},[-1,1]\}$. 

Our variant of the general simplicial approximation Theorem \ref{thm:g-s-a-th} is the following.

\begin{thm}[Weakly simplicial approximation] \label{thm:zeeman}
Let $K$ and $L$ be locally finite simplicial complexes and let $F:|K|\to|L|$ be a continuous map. Assume $|L|\subset\R^q$. Then, for each strictly positive continuous fun\-ction $\veps:|K|\to\R^+$, there exist a subdivision $K'$ of $K$ and a weakly simplicial map $F^*:|K'|\to|L|$ such that 
\[
\text{$\|F^*(w)-F(w)\|_q<\veps(w)\,$ for each $w\in|K'|=|K|$}.
\]
In addition, if $w\in|K|$ and $\tau$ is the carrier of $F(w)$ in $L$, then $F^*(w)\in\tau$.
\end{thm}
\begin{proof}
The proof is conducted in several steps:

\noindent{\sc Step I.} {\em Initial preparation.} 
Assume the simplicial complex $K$ is infinite, because if $K$ is finite the result follows from the classical simplicial approximation Theorem \ref{thm:c-s-a-th}.

Denote $P:=|K|\subset\R^p$ the realization of $K$. It turns out that $P$ is locally compact, but not compact. Choose a sequence $\{P_n\}_{n\in\N}$ of compact subsets of $P$ such that for each $n\in\N^*$:
\begin{itemize}
\item $P_n:=|K_n|$ is the realization of a finite subcomplex $K_n$ of $K$.
\item $\Int_P(P_n)$ is compatible with $K_n$, that is, it is the union of the interiors of some of the simplices of $K_n$.
\item $P_{n-1}\subsetneq\Int_P(P_n)$, where $P_0:=\emptyset$.
\item $\bigcup_{n\in\N}P_n=P$.
\end{itemize}

The compact sets $P_n$ can be constructed as follows. Let $\theta:\R^p\to\R$ be a continuous function such that $\theta^{-1}(0)=\cl_{\R^n}(P)\setminus P$ and consider the map $\theta^*:\R^p\setminus\theta^{-1}(0)\to\R^{p+1}$, $x\mapsto(x,1/\theta(x))$, which is a homeomorphism onto its image and satisfies $\theta^*(P)$ is a closed subset of $\R^{p+1}$. For each $r>0$ denote $\Bb(r):=P\cap(\theta^*)^{-1}(B_{p+1}(0,r))$ and $\overline{\Bb}(r):=P\cap(\theta^*)^{-1}(\overline{B}_{p+1}(0,r))$, where $B_{p+1}(0,r)$ is the open ball of $\R^{p+1}$ with center the origin and radius $r$ and $\overline{B}_{p+1}(0,r)$ is its closure in $\R^{p+1}$. Take a strictly increasing sequence of natural numbers $\{m_n\}_{n\in\N^*}$ such that $\ol{\Bb}(m_1)\neq\emptyset$ and consider the collection of compact subsets $\{\ol{\Bb}(m_n)\}_{n\in\N^*}$ of $P$. 

For each $n\in\N^*$ define ${\mathcal V}_n$ as the collection of the vertices $v$ of $K$ whose open stars $\mr{St}(v,K)$ meet $\overline{\Bb}(m_n)$. Let $K_n$ be the subcomplex of $K$ consisting of all simplices $\sigma$ such that $\mc{V}_n$ contains a vertex of $\sigma$, and all their faces. Define $P_n:=|K_n|$. Observe that $P_n=\bigcup_{v\in\mc{V}_n}\ol{\mr{St}}(v,K)$ and $\overline{\Bb}(m_n)\subset\Int_P(P_n)$. We may assume that $P_n\subset\Bb(m_{n+1})$ (changing $m_{n+1}$ by a bigger integer if necessary). Set $K_0:=\emptyset$, $P_0:=|K_0|=\emptyset$ and $P_{-1}:=\emptyset$. The compact sets $P_n$ satisfy the required conditions. Only the second property requires some comments. Pick a point $x\in\Int_P(P_n)$ and let $\sigma\in K_n$ be the carrier of $x$. Let us check: $\sigma^0\subset\Int_P(P_n)$. Once this is proved, $\Int_P(P_n)$ is the union of the interiors of some of the simplices of $K_n$.

The star ${\mr{St}}(\sigma,K)$ is the union of the interiors of all the simplices $\lambda\in K$ that have $\sigma$ as one of their faces. As $x\in\Int_P(P_n)$, there exists an open ball $B_p(x,\delta)$ centered at $x$ and radius $\delta>0$ such that $B_p(x,\delta)\cap P\subset P_n$. Observe that $B_p(x,\delta)\cap P$ meet all the simplices $\lambda\in K$ that have $\sigma$ as one of their faces. As $P_n$ is the realization of the simplicial subcomplex $K_n$, the simplices $\lambda\in K$ that have $\sigma$ as one of their faces belong to $K_n$. Consequently, ${\mr{St}}(\sigma,K)\subset P_n$ and as it is an open subset of $P$, we conclude $\sigma^0\subset{\mr{St}}(\sigma,K)\subset\Int_P(P_n)$.

For each $n\in\N$ let $U_n$ be an open subset of $P$ such that 
\[
P_n\subsetneq U_n\subset\cl_P(U_n)\subset\Int_P(P_{n+1}).
\]
We set $U_{-1}=U_{-2}:=\emptyset$. Let $\{\epsilon_n\}_{n\in\N}$ be the non-increasing sequence of positive real numbers
\[
\text{$\epsilon_n:=\min_{w\in P_n}\{\veps(w)\}>0\,$ for each $n\in\N^*$}
\]
and set $\epsilon_0:=\epsilon_1$.

\noindent{\sc Step II.} {\em Construction of a suitable covering.} 
For each $n\in\N$ let $L_n:=\mr{sd}^{\ell_n}(L)$ be an iterated barycentric subdivision of $L$ such that $\ell_n\leq\ell_{n+1}$ and:
\begin{equation}\label{eq:a}
\text{\em The diameters of all simplices $\xi$ of $L_n$ with $\xi\cap F(P_n\setminus\Int_P(P_{n-1}))\neq\emptyset$ are $<\epsilon_n$\em.}
\end{equation}

For each vertex $v$ of $L_n$ consider the open star $\mr{St}(v,L_n)$ of $v$ in $L_n$ and define the open subset $V_{n,v}$ of $|K|$ given by 
\[
V_{n,v}:=F^{-1}(\mr{St}(v,L_n))\cap(\Int_P(P_n)\setminus\cl_P(U_{n-2})).
\]
We claim: \em ${\mathcal A}:=\{V_{n,v}:\ n\in\N,\ v \text{ is a vertex of $L_n$}\}$ is an open covering of $P$\em.

Pick a point $x_0\in P$. Then there exists $n\in\N$ such that 
\[
x_0\in\Int(P_n)\setminus\Int(P_{n-1})\subset\Int_P(P_n)\setminus\cl_P(U_{n-2}).
\]
In addition, $F(x_0)\in|L|=|L_n|$. As $\{\mr{St}(v,L_n):\ v \text{ is a vertex of $L_n$}\}$ is an open covering of $|L|$, there exists a vertex $v$ of $L_n$ such that $F(x_0)\in\mr{St}(v,L_n)$, so $x_0\in V_{n,v}$, as claimed.
 
By Theorem \ref{thm:refine} there exists a subdivision $K'$ of $K$ such that the collection of closed stars $\ol{\mr{St}}(u,K')$, where $u$ ranges over the vertices of $K'$, refines ${\mathcal A}$. Observe that $K'$ induces subdivisions $K_n'$ of $K_n$ for each $n\in\N^*$.

\noindent{\sc Step III.} {\em Construction of the weakly simplicial map.}
For each $n\in\N$, let us consider the finite subcomplex $T_n$ of $K$ defined by 
\[
T_n:=\{\sigma\in K_n':\ \sigma\subset P_n\setminus\Int_P(P_{n-1})\}.
\]
We claim: {\it $|T_n|=P_n\setminus\Int_P(P_{n-1})$ and consequently $|T_{n-1}|\cap|T_n|=P_{n-1}\setminus\Int_P(P_{n-1})$}. 

Let $x\in P_n\setminus\Int_P(P_{n-1})$, let $\sigma$ be the carrier of $x$ in $K_n'$ and let $\tau$ be the carrier of $x$ in $K_n$. It holds $\sigma\subset\tau$. It is enough to check: $\tau\subset P_n\setminus\Int_P(P_{n-1})$. As $P_n=|K_n|$, we have $\tau\in K_n$, so $\tau\subset P_n$. As $\Int_P(P_{n-1})$ is the union of the interiors of some of the simplices of $K_{n-1}$, either $\tau^0\subset\Int_P(P_{n-1})$ or $\tau\cap\Int_P(P_{n-1})=\varnothing$. As $x\in\tau^0\setminus\Int_P(P_{n-1})$, we conclude $\tau\cap\Int_P(P_{n-1})=\varnothing$ and the claim follows.

We construct inductively weakly simplicial maps $F_n^*:|T_n|\to|L_n|$ satisfying: \em $F_n^*|_{|T_n|\cap|T_{n-1}|}=F_{n-1}^*|_{|T_n|\cap|T_{n-1}|}$ and for each vertex $u$ of $T_n$ it holds: 
\begin{itemize}
\item If $u\in\Int_P(P_n)\setminus\Int_P(P_{n-1})$\em, we have two possibilities: \em either $F_n^*(u)$ is a vertex of $L_n$ such that $F(\ol{\mr{St}}(u,K'))\subset\mr{St}(F_n^*(u),L_n)$ or it is a vertex of $L_{n+1}$ such that $F(\ol{\mr{St}}(u,K'))\subset\mr{St}(F_n^*(u),L_{n+1})$. 
\item If $u\in P_n\setminus\Int_P(P_n)$\em, we have only one possibility: \em $F_n^*(u)$ is a vertex of $L_{n+1}$ such that $F(\ol{\mr{St}}(u,K'))\subset\mr{St}(F_n^*(u),L_{n+1})$.
\end{itemize}\em

Fix a vertex $u$ of $T_n$ and consider the following two cases:

\noindent{\sc Case 1.} If $u\in P_n\setminus\Int_P(P_n)$, there exists a vertex $v$ of $L_{n+1}$ such that $F(\ol{\mr{St}}(u,K'))\subset\mr{St}(v,L_{n+1})$.

\noindent{\sc Case 2.} If $u\in\Int_P(P_n)\setminus\Int_P(P_{n-1})$, then $u$ is a point of $\Int_P(P_n)\setminus\cl_P(U_{n-2})$. In addition, $u$ can also belong to $\Int_P(P_{n+1})\setminus\cl_P(U_{n-1})$ because $\Int_P(P_n)\setminus\cl_P(U_{n-1})\neq\emptyset$. In any case, there exist: 
\begin{itemize}
\item a vertex $v\in L_n$ such that $F(\ol{\mr{St}}(u,K'))\subset\mr{St}(v,L_n)$ and/or
\item a vertex $v'\in L_{n+1}$ such that $F(\ol{\mr{St}}(u,K'))\subset\mr{St}(v',L_{n+1})$.
\end{itemize} 
Consequently, with this procedure we cannot construct a priori a simplicial map, because the vertices $v,v'$ could not belong to the same iterated barycentric subdivision of $L$. This is why we construct inductively a weakly simplicial map.

As $|T_0|=P_0=\varnothing$, in the first induction step we have nothing to do. Let $n\geq1$ and assume we have already constructed the map $F_{n-1}^*$ satisfying the required conditions. We construct next the map $F_n^*$ and to that end we define first $F_n^*$ on the vertices of $T_n$. 

Fix $u$ a vertex of $T_n$ and suppose first $u\in P_{n-1}\setminus\Int_P(P_{n-1})$. As
\[
P_{n-1}\setminus\Int_P(P_{n-1})=|T_{n-1}|\cap|T_n|, 
\]
$u$ is a vertex of the simplicial subcomplex $T_{n-1}\cap T_n$. We define $F_n^*(u):=F_{n-1}^*(u)$. By induction hypothesis $F_n^*(u)=F_{n-1}^*(u)$ is a vertex of $L_n$ such that $F(\ol{\mr{St}}(u,K'))\subset\mr{St}(F_n^*(u),L_n)$.

Suppose next $u\in\Int_P(P_n)\setminus P_{n-1}$. As $\Int_P(P_n)\setminus P_{n-1}\subset\Int_P(P_n)\setminus\Int_P(P_{n-1})$, there exists a vertex $v\in L_n$ such that $F(\ol{\mr{St}}(u,K'))\subset\mr{St}(v,L_n)$ and/or there exists a vertex $v'\in L_{n+1}$ such that $F(\ol{\mr{St}}(u,K'))\subset\mr{St}(v',L_{n+1})$ (as pointed out above). Choose one of the mentioned vertices $v$ or $v'$ and define either $F_n^*(u):=v$ or $F_n^*(u):=v'$.

Finally, if $u\in P_n\setminus\Int_P(P_n)$, we choose a vertex $v\in L_{n+1}$ such that $F(\ol{\mr{St}}(u,K'))\subset\mr{St}(v,L_{n+1})$ and define $F_n^*(u)=v$. 

{\it Pick a simplex $\sigma$ of $T_n$.} We claim: \em if $\sigma$ has vertices $u_1,\ldots,u_r$, there exists a simplex $\xi\in L_n$ such that the points $v_1:=F_n^*(u_1),\ldots,v_r:=F_n^*(u_r)$ belong to $\xi$\em. 

If $v$ is a vertex of $L_n$, then $\mr{St}(v,L_{n+1})\subset\mr{St}(v,L_n)$. After rearranging the indices if necessary, we assume that for some $s\in\{1,\ldots,r\}$ it holds:
\begin{itemize}
 \item $\{v_\ell\}\in L_n$ and $F(\ol{\mr{St}}(u_\ell,K'))\subset\mr{St}(v_\ell,L_n)$ if $\ell\in\{1,\ldots,s\}$,
 \item $\{v_\ell\}\in L_{n+1}\setminus L_n$ and $F(\ol{\mr{St}}(u_\ell,K'))\subset\mr{St}(v_\ell,L_{n+1})$ if $\ell\in\{s+1,\ldots,r\}$,
\end{itemize}
where the latter case is omitted if $s=r$. 

Pick a point $x\in\sigma^0$ and let $\xi$ be the carrier of $F(x)$ in $L_n$. As $x\in\bigcap_{\ell=1}^r\ol{\mr{St}}(u_\ell,K')$, it holds
\[
F(x)\in\bigcap_{\ell=1}^rF(\ol{\mr{St}}(u_\ell,K'))\subset\bigcap_{\ell=1}^s\mr{St}(v_\ell,L_n)\cap\bigcap_{\ell=s+1}^r\mr{St}(v_\ell,L_{n+1}).
\]
Thus, $v_1,\ldots,v_s$ are vertices of $\xi$ and $v_{s+1},\ldots,v_r$ are vertices of the iterated barycentric subdivision $\mr{sd}^{\ell_{n+1}-\ell_n}(\widehat{\xi}\,)$ of the simplicial complex $\widehat{\xi}$ constituted by the simplex $\xi$ and all its faces. Consequently, $v_1,\ldots,v_r\in\xi$, as claimed. 

We keep the notations already introduced and define $F_n^*:|T_n|\to|L_n|$ (simplex by simplex) as one can expect: \em Let $\lambda_1,\ldots,\lambda_r>0$ be such that $x=\sum_{i=1}^r\lambda_iu_i\in\sigma^0$ (where $u_1,\ldots,u_r$ are the vertices of $\sigma$) and $\sum_{i=1}^r\lambda_i=1$. Then
\[
F_n^*(x):=\sum_{i=1}^r\lambda_iF_n^*(u_i)\in\xi.
\]\em
Thus, $F_n^*$ transforms (affinely) each simplex of $T_n$ onto a convex polyhedron contained in a simplex of $L_n$ (if $s<r$ we cannot assure that $F_n^*(\sigma)$ is a simplex because $v_{s+1},\ldots,v_r$ are not vertices of $\xi$). In addition, {\it if $\tau$ is the carrier of $F(x)$ in $L$, then $F_n^*(x)\in\xi\subset\tau$}. Define 
\[
F^*:|K|\to|L|,\ x\mapsto F_n^*(x)\quad\text{if $x\in T_n$.}
\]
The previous map is well-defined, continuous and weakly simplicial because $F_n^*|_{|T_n|\cap|T_{n-1}|}=F_{n-1}^*|_{|T_n|\cap|T_{n-1}|}$ and each $F_n^*$ is continuous and weakly simplicial. By construction $F^*(x)$ belongs to the carrier $\tau$ of $F(x)$ in $L$ for each $x\in|K|$.

\noindent{\sc Step IV.} {\em Approximation.}
Pick $x\in|T_n|=P_n\setminus\Int_P(P_{n-1})$. Let $\sigma$ be the carrier of $x$ in $K'$ (or equivalently in $K'_n$) and let $\xi$ be the carrier of $F(x)$ in $L_n$. The intersection $\xi\cap F(P_n\setminus\Int_P(P_{n-1}))$ is non-empty because it contains $F(x)$. By \eqref{eq:a} the diameter of $\xi$ is strictly smaller than $\epsilon_n$. As $F^*(x)\in\xi$, we have $\|F^*(x)-F(x)\|_q<\epsilon_n\leq\veps(x)$, as required.
\end{proof}

\begin{remarks}\label{rem:rel}
(i) {\it If in the statement of Theorem \em \ref{thm:zeeman} \em $H$ is a subcomplex of $K$ such that $F(|H|)$ is contained in the set of vertices of $L$, then $F^*|_{|H|}=F|_{|H|}$}. 

Let $C$ be a connected component of $|H|$ and let $v$ be the vertex of $L$ such that $F(C)=\{v\}$. If $x\in C$, then the carrier of $F(x)$ is $v$, so the last assertion in Theorem \ref{thm:zeeman} implies $F^*(x)=v=F(x)$. Consequently $F^*|_{|H|}=F|_{|H|}$, as claimed. 

(ii) {\it If $F$ is proper in the statement of Theorem \em \ref{thm:zeeman}\em, then it is well-known that $F^*$ can be chosen simplicial}. 

This result can be proven by combining the version of simplicial approximation theorem presented in \cite[Ch.5, p.223]{b} with the following elementary fact that follows from Lemma \ref{lem:approx}: \em if $F:|K|\to|L|$ is a proper continuous map and $\veps:|K|\to\R$ is a strictly positive function, then there exists a strictly positive function $\delta:|L|\to\R$ such that $\delta(F(x))<\veps(x)$ for each $x\in|K|$\em. $\sqbullet$
\end{remarks}

%%%

\subsection{The `shrink-widen' covering and approximation technique}\label{swt}

Let $\sigma$ be a simplex of $\R^p$, let ${\rm Bd}(\sigma)$ be the boundary of $\sigma$ and let $\sigma^0$ be the interior of $\sigma$. Recall that ${\rm Bd}(\sigma)$ is the union of proper faces of $\sigma$ and $\sigma^0$ is the open simplex of $\R^p$ such that $\sigma^0=\sigma\setminus{\rm Bd}(\sigma)$. Let $b_\sigma$ be the barycenter of $\sigma$. Given $\veps\in(0,1)$ denote $h_\veps:\R^p\to\R^p,\ x\mapsto b_\sigma+(1-\veps)(x-b_\sigma)$ the homothety of $\R^p$ of center $b_\sigma$ and ratio $1-\veps$ and define the \em $(1-\veps)$-shrinking $\sigma^0_\veps$ of $\sigma^0$ \em by $\sigma^0_\veps:=h_\veps(\sigma^0)$. Note that $\cl_{\R^p}(\sigma^0_\veps)=h_\veps(\sigma)\subset\sigma^0$ for each $\veps\in(0,1)$ and $\sigma^0_\veps$ tends to $\sigma^0$ when $\veps\to 0$. In addition, $\sigma^0=\bigcup_{\veps\in (0,1)}\sigma^0_\veps$ and $\sigma^0_{\veps_2}\subset\sigma^0_{\veps_1}$ if $0<\veps_1\leq\veps_2<1$.

We fix the following notations for the rest of the subsection. Let $r\in\N^*\cup\{\infty\}$ and let $X\subset\R^m$ be a locally compact set. Suppose $X$ is \em $\mscr{C}^r$ triangulable on open simplices\em, that is, there exists a locally finite simplicial complex $K$ of some $\R^p$ and a homeomorphism $\Phi:|K|\to X$ such that: \em the set $\Phi(\sigma^0)$ is a $\mscr{C}^r$ submanifold of $\R^m$ and the restriction $\Phi|_{\sigma^0}:\sigma^0\to\Phi(\sigma^0)$ is a $\mscr{C}^r$ diffeomorphism for each open simplex $\sigma^0$ of $K$\em. Define $\Kk^0:=\{\Phi(\sigma^0)\}_{\sigma\in K}$ and $\Kk:=\{\Phi(\sigma)\}_{\sigma\in K}$. 

To lighten the notation the elements of $\Kk$ will be denoted with the letters $\sfs,\sft,\ldots$ while those of $\Kk^0$ with the letters $\sfs^0,\sft^0,\ldots$ in such a way that $\cl_{\R^m}(\sfs^0)=\sfs$. In other words, if $\sfs=\Phi(\sigma)$, then $\sfs^0=\Phi(\sigma^0)$. Moreover, we indicate $\sfs^0_\veps$ the $(1-\veps)$-shrinking of $\sfs^0=\Phi(\sigma^0)$ corresponding to $\sigma^0_\veps$ via $\Phi$, that is, $\sfs^0_\veps:=\Phi(\sigma^0_\veps)$.

Consider a $\mscr{C}^r$ tubular neighborhood $\rho_{\sfs^0}:T_{\sfs^0}\to\sfs^0$ of $\sfs^0$ in $\R^m$ and for each $\eta>0$ the open subset $T_{\sfs^0,\eta}:=\{x\in T_{\sfs^0}:\ \|x-\rho_{\sfs^0}(x)\|_m<\eta\}$ of $\R^m$. We write $\sfs^0_{\veps,\eta}$ to denote the \em $\eta$-widening of $\sfs^0_\veps$ with respect to $\rho_{\sfs^0}$\em, which is the open neighborhood $\sfs^0_{\veps,\eta}:=(\rho_{\sfs^0})^{-1}(\sfs^0_\veps)\cap T_{\sfs^0,\eta}$ of $\sfs^0_\veps$ in $\R^m$. If $C$ is a closed subset of $\R^m$ such that $C\cap\cl_{\R^m}(\sfs^0_\veps)=\varnothing$, there exists $\eta>0$ such that $C\cap\cl_{\R^m}(\sfs^0_{\veps,\eta})=\varnothing$ (recall that $\cl_{\R^m}(\sfs^0_\veps)$ is compact). Denote $\rho_{\sfs^0,\veps,\eta}:=\rho_{\sfs^0}|_{\sfs^0_{\veps,\eta}}:\sfs^0_{\veps,\eta}\cap X\to\sfs^0_\veps$ the $\mscr{C}^r$ retraction obtained restricting $\rho_{\sfs^0}$ from $\sfs^0_{\veps,\eta}\cap X$ to $\sfs^0_\veps$.

\begin{lem}\label{lem:covering}
Fix a strictly positive function $\eta:\Kk^0\to\R^+$. Then for each $\sfs^0\in\Kk^0$ there exist a non-empty open subset $V_{\sfs^0}$ of~$\sfs^0$ (a `shrinking' of $\sfs^0$), an open neighborhood $U_{\sfs^0}$ of $V_{\sfs^0}$ in $X$ (a `widening' of~$V_{\sfs^0}$) satisfying $V_{\sfs^0}=U_{\sfs^0}\cap\sfs^0$ and a $\mscr{C}^r$ retraction $r_{\sfs^0}:U_{\sfs^0}\to V_{\sfs^0}$ such that:
\begin{itemize}
\item[(i)] $\{U_{\sfs^0}\}_{\sfs^0\in\Kk^0}$ is a locally finite open covering of $X$.
\item[(ii)] $\cl_{\R^m}(U_{\sfs^0})\cap\sft=\varnothing$ for each pair $(\sfs^0,\sft)\in\Kk^0\times\Kk$ satisfying ${\sfs^0}\cap\sft=\varnothing$.
\item[(iii)] $\sup_{x\in U_{\sfs^0}}\{\|x-r_{\sfs^0}(x)\|_m\}<\eta(\sfs^0)$ for each $\sfs^0\in\Kk^0$.
\end{itemize}
\end{lem}
\begin{proof}
Define $d:=\max\{\dim(\sfs^0):\ \sfs^0\in\Kk^0\}\leq m$, where $\dim(\sfs^0)$ is the dimension of $\sfs^0$ as a $\mscr{C}^r$~submanifold of $\R^m$. Of course $d$ coincides with the dimension of the semialgebraic set $|K|$, which is equal to $\max\{\dim(\sigma^0):\ \sigma\in K\}$. Let $\Kk^0_e:=\{\sfs^0\in\Kk^0:\, \dim(\sfs^0)\leq e\}$ for $e\in\{0,1,\ldots,d\}$. Let us prove by induction on $e\in\{0,1,\ldots,d\}$ that: \em For each $\sfs^0\in\Kk^0_e$ there exist an open subset $U_{\sfs^0}^e$ of $X$ and a $\mscr{C}^r$ retraction $r_{\sfs^0}^e:U_{\sfs^0}^e\to V_{\sfs^0}^e:=U_{\sfs^0}^e\cap\sfs^0\neq\varnothing$ such that:
\begin{itemize}
\item[$(\mr{a})$] $\bigcup_{\sfs^0\in\Kk^0_e}\sfs^0\subset\bigcup_{\sfs^0\in\Kk^0_e}U_{\sfs^0}^e$.
\item[$(\mr{b})$] $\cl_{\R^p}(U_{\sfs^0}^e)\cap\sft=\varnothing$ for each pair $(\sfs^0,\sft)\in\Kk^0_e\times\Kk$ satisfying ${\sfs^0}\cap\sft=\varnothing$.
\item[$(\mr{c})$] $\sup_{x\in U_{\sfs^0}^e}\{\|x-r_{\sfs^0}^e(x)\|_m\}<\eta(\sfs^0)$ for each $\sfs^0\in\Kk^0_e$.
\end{itemize}\em

Suppose first $e=0$. Choose $\{v\}\in\Kk^0_0$. As the family $\Kk$ is locally finite in $X$, the union $\bigcup_{\sft\in\Kk,v\not\in\sft}\sft$ is closed in $X$ and it does not contain $v$. Consequently, there exists $\eta'_{v}\in(0,\eta(\{v\}))$ such that the open ball $B(v,2\eta'_v)$ of~$\R^m$ of center $v$ and radius $2\eta'_v$ does not meet $\bigcup_{\sft\in\Kk,v\not\in\sft}\sft$. Define~$U_{\{v\}}^0:=B(v,\eta'_v)\cap X$, $V_{\{v\}}^0:=\{v\}$ and $r_{\{v\}}^0:U_{\{v\}}^0\to V_{\{v\}}^0,\ x\mapsto v$ the constant map for each $\{v\}\in\Kk^0_0$.

Fix $e\in\{0,\ldots,d-1\}$ and suppose that the assertion is true for such an $e$. Pick $\sigma\in K$ of dimension $e+1$ and consider the compact subset
\[
C_\sigma:=\sigma\setminus\Phi^{-1}\Big(\bigcup_{\tau\in K,\tau\subset\mathrm{Bd}(\sigma)}U^e_{\Phi(\tau^0)}\Big)
\]
of $\sigma^0=\bigcup_{\veps\in (0,1)}\sigma^0_\veps$. Let $\veps(\sigma^0)\in (0,1)$ be such that $C_\sigma\subset\sigma^0_{\veps(\sigma^0)}$. If $\sfs=\Phi(\sigma)$, define $\veps(\sfs^0):=\veps(\sigma^0)$ and $\sfs^0_{\veps(\sfs^0)}:=\Phi(\sigma^0_{\veps(\sigma^0)})$. We have
\[
\textstyle
\bigcup_{\sfs^0\in\Kk^0_{e+1}}\sfs^0\subset\bigcup_{\sfs^0\in\Kk^0_e}
U_{\sfs^0}^e\cup\bigcup_{\sfs^0\in\Kk^0_{e+1}\setminus\Kk^0_e}
\sfs^0_{\veps(\sfs^0)}.
\]
If $(\sfs^0,\sft)\in(\Kk^0_{e+1}\setminus\Kk^0_e)\times\Kk$ satisfies $\sfs^0\cap\sft=\varnothing$, then $\cl_{\R^m}(\sfs^0_{\veps(\sfs^0)})\cap\sft=\varnothing$ because $\cl_{\R^m}(\sfs^0_{\veps(\sfs^0)})\subset\sfs^0$. Let $\sfs^0\in\Kk^0_{e+1}\setminus\Kk^0_e$. As the family $\Kk$ is locally finite in $X$, there exists $\eta'(\sfs^0)\in(0,\eta(\sfs^0))$ such that
\[
\cl_{\R^m}\!\big(\sfs^0_{\veps(\sfs^0),\eta'(\sfs^0)}\cap X\big)\cap\sft=\varnothing
\]
for each pair $(\sfs^0,\sft)\in(\Kk^0_{e+1}\setminus\Kk^0_e)\times\Kk$ satisfying $\sfs^0\cap\sft=\varnothing$. For each $\sfs^0\in\Kk^0_{e+1}$ define:
\begin{itemize}
\item $V_{\sfs^0}^{e+1}:=V_{\sfs^0}^e$, $U_{\sfs^0}^{e+1}:=U_{\sfs^0}^e$ and $r_{\sfs^0}^{e+1}:=r_{\sfs^0}^e$ if $\sfs^0\in\Kk^0_e$ and \vspace{.3em}
\item $V_{\sfs^0}^{e+1}:=\sfs^0_{\veps(\sfs^0)}$, $U_{\sfs^0}^{e+1}:=\sfs^0_{\veps(\sfs^0),\eta'(\sfs^0)}\cap X$ and $r_{\sfs^0}^{e+1}:=\rho_{\sfs^0,\veps(\sfs^0),\eta'(\sfs^0)}$ if $\sfs^0\in\Kk^0_{e+1}\setminus\Kk^0_e$.
\end{itemize}

The open sets $U_{\sfs^0}^{e+1}$, the non-empty sets $V_{\sfs^0}^{e+1}$ and the retractions $r_{\sfs^0}^{e+1}:U_{\sfs^0}^{e+1}\to V_{\sfs^0}^{e+1}$ for $\sfs^0\in\Kk^0_{e+1}$ satisfy conditions $(\mr{a})$ to $(\mr{c})$, as required.

Define the open subsets $U_{\sfs^0}:=U_{\sfs^0}^d$ of $X$ and the $\mscr{C}^r$ retractions $r_{\sfs^0}:=r_{\sfs^0}^d$ for each $\sfs^0\in\Kk^0_d=\Kk^0$. Evidently properties $(\mr{ii})$ and $(\mr{iii})$ hold and the family $\{U_{\sfs^0}\}_{\sfs^0\in\Kk^0}$ is a covering of $X$. It remains to show that such a family is locally finite in $X$. Let $x\in X$ and let $\sfu^0$ be the unique element of $\Kk^0$ such that $x\in\sfu^0$. For each $\sft\in\Kk$ and each $\sfs^0\in\Kk^0$ define the finite set $I_\sft:=\{\sfs^0\in\Kk^0:\ \sfs^0\subset\sft\}$ and the set $J_{\sfs^0}:=\{\sft\in\Kk:\ \sfs^0\subset\sft\}$. As the family $\Kk$ is locally finite in $X$, each set $J_{\sfs^0}$ is finite as well. If $\sft\in\Kk$ and $\sfs^0\in\Kk^0$ satisfies $U_{\sfs^0}\cap\sft\neq\varnothing$, then $\sft\in J_{\sfs^0}$ by property (ii). In particular $U_{\sfs_0}\subset\bigcup_{\sft\in J_{\sfs^0}}\sft$. 

Define the finite set $M_{\sfu^0}:=\bigcup_{\sft\in J_{\sfu^0}}I_\sft$ and the set $N_{\sfu^0}:=\{\sfs^0\in\Kk^0:\ U_{\sfu^0}\cap U_{\sfs^0}\neq\varnothing\}$. Let us show: \em $N_{\sfu^0}$ is finite \em by showing that $N_{\sfu^0}\subset M_{\sfu^0}$. This will complete the proof. If $\sfs^0\in N_{\sfu^0}$, then 
\[\textstyle
\varnothing\neq U_{\sfu^0}\cap U_{\sfs^0}\subset \bigcup_{\sft\in J_{\sfs^0}}\big(U_{\sfu^0}\cap\sft\big).
\]
Thus, there exists $\sft\in J_{\sfs^0}$ such that $U_{\sfu^0}\cap\sft\neq\varnothing$, so $\sfs^0\in I_\sft$ and $\sft\in J_{\sfu^0}$, that is, $\sfs^0\in M_{\sfu^0}$, as required.
\end{proof}

\begin{lem}\label{lem:approximation'}
Let $L$ be a locally finite simplicial complex of $\R^q$ and let $g\in\mscr{C}^0(X,|L|)$. Suppose that for each $\sft\in\Kk$ the restriction $g|_{\sft^0}$ belongs to $\mscr{C}^r(\sft^0,|L|)$ and there exists $\xi_\sft\in L$ such that $g(\sft)\subset\xi_\sft$. Then for each strictly positive continuous function $\delta:X\to\R^+$, there exists $h\in\mscr{C}^r(X,|L|)$ with the following properties:
\begin{itemize}
\item[(i)] For each $\sft\in\Kk$, there exists an open neighborhood $W_\sft$ of $\sft$ in $X$ such that $h(W_\sft)\subset\xi_\sft$.
\item[(ii)] $\|h(x)-g(x)\|_q<\delta(x)$ for each $x \in X$.
\end{itemize} 
\end{lem}

To prove this lemma and Proposition \ref{wr} below we need the following basic topological result that we borrow from \cite[Lem.2.4]{abf}.

\begin{lem}\label{neighs}
Let $T$ be a paracompact topological space, let $\{T_k\}_{k\in\N}$ be a locally finite family of subsets of $T$ and for each $k\in\N$ let $V_k\subset T$ be an open neighborhood of $T_k$. Then there exist open neighborhoods $U_k\subset T$ of $T_k$ such that $U_k\subset V_k$ for each $k\in\N$ and the family $\{U_k\}_{k\in\N}$ is locally finite in~$T$.
\end{lem}
\begin{proof}
For each $x\in T$ let $B_x\subset T$ be an open neighborhood of $x$ that meets only finitely many~$T_k$. The family $\{B_x\}_{x\in T}$ is an open covering of $T$. As $T$ is paracompact, there exists a locally finite open covering $\{W_\ell\}_{\ell\in L}$ of $T$, which is a refinement of $\{B_x\}_{x\in T}$. Observe that each $W_\ell$ meets only finitely many~$T_k$. For each $k\in\N$ define $U_k':=\bigcup_{W_\ell\cap T_k\neq\varnothing}W_\ell$. Note that $T_k\subset U_k'$ for each $k\in\N$. We claim: \em The family $\{U_k'\}_{k\in\N}$ is locally finite in $T$\em. 

Fix a point $x\in T$ and consider a neighborhood $V_x\subset T$ of $x$ that meets finitely many $W_\ell$, say $W_{\ell_1},\ldots,W_{\ell_r}$. The union $\bigcup_{j=1}^rW_{\ell_j}$ meets only finitely many $T_k$, say $T_{k_1},\ldots,T_{k_s}$. If $k\not\in\{k_1,\ldots,k_s\}$, the intersection $U_k'\cap V_x=\varnothing$. To finish it is enough to define $U_k:=U_k'\cap V_k$ for each $k\in\N$.
\end{proof}

We are ready to prove Lemma \ref{lem:approximation'}.

\begin{proof}[Proof of Lemma \em\ref{lem:approximation'}]
We will give the proof only in the case $X$ is non-compact (because if $X$ is compact, the proof is similar, but easier). Choose a sequence $\{X_n\}_{n\in\N}$ of compact subsets of $X$ such that for each $n\in\N$:
\begin{itemize}
 \item $X_n$ is a finite union of elements of $\Kk$, say $X_n=\bigcup_{\sft\in\Kk_n}\sft$ for some finite set $\Kk_n\subset\Kk$.
 \item $X_{n-1}\subsetneq\Int_X(X_n)$, where $X_{-1}:=\varnothing$.
 \item $\bigcup_{n\in\N}X_n=X$.
\end{itemize}

Let $\{\delta_n\}_{n\in\N}$ be the decreasing sequence of positive real numbers defined by
\[
\text{$\delta_n:=\min_{x\in X_n}\{\delta(x)\}>0\,$ for each $n\in\N$.}
\]
Write $\veps_{-1}:=1$. As the restriction of $g|_{X_n}$ is uniformly continuous, for each $n\in\N$ there exists $\veps_n\in(0,\veps_{n-1})$ such that 
\begin{equation}\label{eq:uc}
\text{$\|g(x')-g(x)\|_q<\delta_n\,$ for each pair $x',x\in X_n$ with $\|x'-x\|_m<\veps_n$}.
\end{equation}

Fix $\sfs^0\in\Kk^0$. We claim: \em there exists a unique integer $n:=n(\sfs^0)\in\N$ such that $\sfs^0\subset X_n\setminus X_{n-1}$\em.

Pick $x\in\sfs^0$ and let $k\in\N$ be such that $x\in X_k=\bigcup_{\sft\in\Kk_k}\sft$. Thus, $x\in\sft$ for some $\sft\in\Kk_k$, so $x\in\sfs^0\cap\sft$ and $\sfs^0\subset\sft\subset X_k$. This proves that if $\sfs^0\cap X_k\neq\varnothing$, then $\sfs^0\subset X_k$. Note that $n=n(\sfs^0):=\min\{k\in\N:\ \sfs^0\subset X_k\}$ is the unique natural number such that $\sfs^0\subset X_n\setminus X_{n-1}$, as claimed.

Consider the strictly positive function $\eta:\Kk^0\to\R^+,\ \sfs^0\mapsto\veps_{n(\sfs^0)+1}$. By Lemma \ref{lem:covering} for each $\sfs^0\in\Kk^0$ there exist an open subset $U_{\sfs^0}$ of $X$ with $V_{\sfs^0}:=U_{\sfs^0}\cap {\sfs^0}\neq\varnothing$ and a $\mscr{C}^r$ retraction $r_{\sfs^0}:U_{\sfs^0}\to V_{\sfs^0}$ such that: 
\begin{align}
&\text{$\{U_{\sfs^0}\}_{\sfs^0\in\Kk^0}$ is a locally finite covering of $X$,}\nonumber\\
&\text{$\cl_{\R^m}(U_{\sfs^0})\cap\sft=\varnothing\,$ for each pair $(\sfs^0,\sft)\in\Kk^0\times\Kk$ such that ${\sfs^0}\cap\sft=\varnothing$ and}\label{eq:cl'}\\
&\text{$\textstyle\sup_{x\in U_{\sfs^0}}\{\|x-r_{\sfs^0}(x)\|_m\}<\veps_{n(\sfs^0)+1}\,$ for each $\sfs^0\in\Kk^0$.}\label{eq:max'}
\end{align}

Let $\{\theta_{\sfs^0}:X\to[0,1]\}_{\sfs^0\in\Kk^0}$ be a $\mscr{C}^r$ partition of unity subordinated to the locally finite open covering $\{U_{\sfs^0}\}_{\sfs^0\in\Kk^0}$ of $X$. To prove the existence of such $\mscr{C}^r$ partition of unity one can proceed as follows. We may assume that $X$ is a closed subset of $\R^m$. The family $\{\cl_{\R^m}(U_{\sfs^0})\}_{\sfs^0\in\Kk^0}$ is locally finite in $\R^m$. By Lemma \ref{neighs} there exists a locally finite family $\{\Omega_{\sfs^0}\}_{\sfs^0\in\Kk^0}$ of open subsets of $\R^m$ such that $\cl_{\R^m}(U_{\sfs^0})\subset\Omega_{\sfs^0}$. Let $\Omega_{\sfs^0}'\subset\Omega_{\sfs^0}$ be an open subset such that $U_{\sfs^0}=X\cap\Omega_{\sfs^0}'$. Let $\{\Theta_0\}\cup\{\Theta_{\sfs^0}:X\to[0,1]\}_{\sfs^0\in\Kk^0}$ be a $\mscr{C}^r$ partition of unity subordinated to the locally finite open covering $\{\R^m\setminus X\}\cup\{\Omega_{\sfs^0}'\}_{\sfs^0\in\Kk^0}$ of $\R^m$. Now, it is enough to consider $\theta_{\sfs^0}:=\Theta_{\sfs^0}|_X$ for each $\sfs^0\in\Kk^0$ in order to have the desired $\mscr{C}^r$ partition of unity subordinated to $\{U_{\sfs^0}\}_{\sfs^0\in\Kk^0}$.

For each $\sfs^0\in\Kk^0$ the map
\[
g\circ r_{\sfs^0}:U_{\sfs^0}\to V_{\sfs^0}\subset\sfs^0\subset\sfs\to\xi_\sfs,\ x\mapsto r_{\sfs^0}(x)\mapsto g(r_{\sfs^0}(x))
\]
is a $\mscr{C}^r$ map, so also the map $H_{\sfs^0}:X\to\R^q$ defined by 
\begin{equation*}
H_{\sfs^0}(x):=\begin{cases}
\theta_{\sfs^0}(x)\cdot g(r_{\sfs^0}(x))&\text{ if $x\in U_{\sfs^0}$,}\\
0&\text{ if $x\in X\setminus U_{\sfs^0}$,}
\end{cases}
\end{equation*}
belongs to $\mscr{C}^r(X,\R^q)$. Consider the $\mscr{C}^r$ map $H:=\sum_{\sfs^0\in\Kk^0}H_{\sfs^0}:X\to\R^q$.

Fix $\sft\in\Kk$ and define $W_\sft:=X\setminus\bigcup_{\sfs^0\in\Kk^0,\,{\sfs^0}\cap\sft=\varnothing}\cl_{\R^m}(U_{\sfs^0})$. As the family $\Kk^0$ is locally finite in $X$, we deduce that $W_\sft$ is by \eqref{eq:cl'} an open neighborhood of $\sft$ in $X$. We claim: $H(W_\sft)\subset\xi_\sft$. 

Pick $x\in W_\sft$. If $\sfs^0\in\Kk^0$ and ${\sfs^0}\cap\sft=\varnothing$, then $\theta_{\sfs^0}(x)=0$ because the support of $\theta_{\sfs^0}$ is contained in $U_{\sfs^0}$ and $x\not\in\cl_{\R^m}(U_{\sfs^0})$. If ${\sfs^0}\cap\sft\neq\varnothing$, then ${\sfs^0}\subset\sft$, so we conclude
\begin{align}
&\sum_{\substack{\sfs^0\in\Kk^0,\,{\sfs^0}\subset\sft,\\x\in U_{\sfs^0}}}\theta_{\sfs^0}(x)=1\ \text{and}\label{eq:sum'}\\
H(x)=&\sum_{\substack{\sfs^0\in\Kk^0,\,{\sfs^0}\subset\sft,\\x\in U_{\sfs^0}}}\theta_{\sfs^0}(x)g(r_{\sfs^0}(x)).\label{eq:sum2'}
\end{align}

If $\sfs^0\in\Kk^0$ satisfies ${\sfs^0}\subset\sft$ and $x\in U_{\sfs^0}$, then $r_{\sfs^0}(x)\in V_{\sfs^0}\subset\sfs^0$, so $g(r_{\sfs^0}(x))\in\xi_\sft$. As $\xi_\sft$ is a convex subset of $\R^q$ and each $g(r_{\sfs^0}(x))\in\xi_\sft$ if ${\sfs^0}\subset\sft$ and $x\in U_{\sfs^0}$, we conclude by means of \eqref{eq:sum'} and \eqref{eq:sum2'} that $H(x)\in\xi_\sft$. Consequently, $H(W_\sft)\subset\xi_\sft$, as claimed.

As $X=\bigcup_{\sft\in\Kk}\sft=\bigcup_{\sft\in\Kk}W_\sft$, we deduce $H(X)$ is contained in $|L|$ and $h:X\to|L|,\,x\mapsto H(x)$ is a $\mscr{C}^r$ map that satisfies property (i).

It remains to prove (ii). Fix $x\in X_n\setminus X_{n-1}$ for some $n\in\N$. Denote $\sfu$ the unique element of $\Kk$ such that $x\in\sfu^0$. As $\sfu^0\cap X_n\neq\varnothing$, we have $\sfu^0\subset X_n$. Observe that $\sfu^0\cap X_{n-1}=\varnothing$, because otherwise $x\in\sfu^0\subset X_{n-1}$, which is a contradiction. Thus, $\sfu^0\subset X_n\setminus X_{n-1}$, so $n(\sfu^0)=n$. If $\sfs^0\in\Kk^0$ satisfies $x\in U_{\sfs^0}$, then $U_{\sfs^0}\cap\sfu\neq\varnothing$ and by \eqref{eq:cl'} we have 
\[
\sfs^0\subset\sfu\subset\cl_X(X_n\setminus X_{n-1})\subset X_n\setminus\Int_X(X_{n-1})\subset X_n\setminus X_{n-2}=(X_n\setminus X_{n-1})\sqcup(X_{n-1}\setminus X_{n-2}),
\]
where $X_{-2}:=\varnothing$. Thus, $n(\sfs^0)\in\{n-1,n\}$. Consequently, $r_{\sfs^0}(x)\in\sfs^0\subset\sfu\subset X_n$ and by \eqref{eq:max'} we have $\|x-r_{\sfs^0}(x)\|_m<\veps_{n(\sfs^0)+1}\leq\veps_n$. Now inequality \eqref{eq:uc} implies that
\begin{multline*}
\|h(x)-g(x)\|_q=\Big\|\sum_{\sfs^0\in\Kk^0,\,x\in U_{\sfs^0}}\theta_{\sfs^0}(x)\big(g(r_{\sfs^0}(x))-g(x)\big)\Big\|_q\leq\\
\leq\sum_{\sfs^0\in\Kk^0,\,x\in U_{\sfs^0}}\theta_{\sfs^0}(x)\|g(r_{\sfs^0}(x))-g(x)\|_q<\delta_n\leq\delta(x),
\end{multline*}
as required.
\end{proof}

\begin{remark}\label{rem:constant1}
We keep the notations of the preceding proof. If there exist $\sft\in\Kk$ and $w\in|L|$ such that $g$ takes the constant value $w$ on $\sft$, then by \eqref{eq:sum'} and \eqref{eq:sum2'} the $\mscr{C}^r$ map $h:X\to|L|$ is constant on the open neighborhood $W_\sft\subset X$ of $\sft$ and takes the constant value $w$. In particular, this is always true if $\sft=\{v\}$ for any vertex $v$ of $K$. As a consequence, if $X$ has at least one accumulation point, then $h$ is not injective. $\sqbullet$
\end{remark}

%%%
\subsection{Proof of Theorem \ref{thm:main1}}\label{thm11}
Let $X$ be a locally compact subset of some $\R^m$. We assume $X$ is non-compact. If $X$ is compact the proof is similar, but easier. As $Y\subset\R^n$ is a weakly $\mscr{C}^r$ triangulable set, there exist a locally finite simplicial complex $L$ of some $\R^q$ and a homeomorphism $\Psi:|L|\to Y$ such that $\Psi|_\xi\in\mscr{C}^r(\xi,Y)$ for each $\xi\in L$. In particular, $Y$ is locally compact in $\R^n$. Consider a continuous map $f:X\to Y$ and a strictly positive continuous function $\veps:X\to\R^+$. We will show: \em There exists $\Hh\in\mscr{C}^r(X,Y)$ such that $\|\Hh(x)-f(x)\|_n<\veps(x)$ for each $x\in X$\em.

By the first part of Corollary \ref{cor:closedness}, we can assume $X$ is closed in $\R^m$ and $Y$ is closed in $\R^n$.

The proof is conducted in several steps:

\noindent{\sc Step I.} \em Initial preparation\em. As $X$ is closed in $\R^m$, Tietze's extension theorem guarantees the existence of a strictly positive continuous function $E:\R^m\to\R^+$ and a continuous map $\widehat{f}:\R^m\to\R^n$ such that $E(x)=\veps(x)$ and $\widehat{f}(x)=f(x)$ for each $x\in X$. By \cite[Cor.3.5]{hanner} $Y$ is an absolute neighborhood retract. Consequently, as $Y$ is closed in $\R^n$, there exists an open neighborhood $W\subset\R^n$ of $Y$ and a continuous retraction $\rho:W\to Y$. Consider the open neighborhood $U:=(\widehat{f}\,)^{-1}(W)$ of $X$ in $\R^m$ and the continuous extension $\widetilde{f}:U\to Y,\ x\mapsto\rho(\widehat{f}(x))$ of $f$. Renaming $U$ as $X$, $E|_U$ as $\veps$ and $\widetilde{f}$ as $f$, we can assume that $X$ is an open subset of~$\R^m$, so in particular $X$ is a $\mscr{C}^\infty$ manifold. 

By the Cairns-Whitehead triangulation theorem $X$ is `$\mscr{C}^\infty$ triangulable on open simplices', that is, there exist a locally finite simplicial complex $K$ of some $\R^p$ and a homeomorphism $\Phi:|K|\to X$ such that $\Phi(\sigma^0)$ is a $\mscr{C}^\infty$ submanifold of $\R^m$ and the restriction $\Phi|_{\sigma^0}:\sigma^0\to\Phi(\sigma^0)$ is a $\mscr{C}^{\infty}$ diffeomorphism for each open simplex $\sigma^0$ of $K$ (see \cite[Lemma 3.5 \& p.82]{hu}). Set $P:=|K|$, $\Kk:=\{\Phi(\sigma)\}_{\sigma\in K}$ and $\Kk^0:=\{\Phi(\sigma^0)\}_{\sigma\in K}$. Define $Q:=|L|\subset\R^q$.

\noindent{\sc Step II.} \em Reduction to the weakly simplicial case\em. Choose a sequence $\{X_k\}_{k\in\N}$ of compact subsets of $X$ such that $\bigcup_{k\in\N}X_k=X$ and $X_{k-1}\subsetneq\Int_X(X_k)$ for each $k\in\N$, where $X_{-1}:=\varnothing$. Note that the family $\{X_k\setminus X_{k-1}\}_{k\in\N}$ is locally finite in $X$.

Fix $k\in\N$ and consider the compact subsets $P_k:=\Phi^{-1}(X_k)$ of $P$ and $Q_k:=\Psi^{-1}(f(X_k))$ of $Q$. Define $\epsilon_k:=\min_{x\in X_k}\{\veps(x)\}>0$ and $\mu_k:=\min\{1,\mr{dist}_{\R^q}(Q_k,\cl_{\R^q}(Q)\setminus Q)\}>0$, where $\mr{dist}_{\R^q}(Q_k,\varnothing):=+\infty$. Consider the compact subset $V_k$ of $Q$ defined by
\begin{equation}\label{eq:Vk}
V_k:=\{z\in Q:\ \mr{dist}_{\R^q}(z,Q_k)\leq\mu_k/2\}.
\end{equation}
By the uniform continuity of $\Psi$ on $V_k$ there exists (for each $k\in\N$) $\delta_k>0$ such that
\begin{equation}\label{eq:uc-Psi}
\|\Psi(z')-\Psi(z)\|_n<\epsilon_k\;\text{ for each pair $z',z\in V_k$ with $\|z'-z\|_q<\delta_k$}.
\end{equation}

Consider the $\mscr{C}^0$ map $F:=\Psi^{-1}\circ f\circ\Phi:|K|=P\to Q=|L|$. Applying Theorem \ref{thm:zeeman} to $F$ we obtain, after replacing $K$ by one of its subdivisions, that there exists a weakly simplicial map $F^*:|K|\to|L|$ such that
\begin{equation}\label{eq:F*}
\text{$\|F^*(w)-F(w)\|_q<\min\{\mu_k/4,\delta_k/2\}\,$ for each $k\in\N$ and each $w\in P_k\setminus P_{k-1}$,}
\end{equation}
where $P_{-1}:=\varnothing$. Define the continuous maps $g:=\Psi^{-1}\circ f=F\circ\Phi^{-1}:X\to|L|$ and $g^*:=F^*\circ\Phi^{-1}:X\to|L|$. For each $\sft\in\Kk$ the restriction $\Phi^{-1}|_{\sft^0}:\sft^0\to\Phi^{-1}(\sft^0)$ is a $\mscr{C}^\infty$ diffeomorphism. Thus, as $F^*|_{\Phi^{-1}(\sft^0)}$ is an affine map, $g^*|_{\sft^0}\in\mscr{C}^\infty(\sft^0,|L|)$. As $F^*$ is weakly simplicial and $\Phi^{-1}(\sft)\in K$, there exists $\xi_\sft\in L$ such that $g^*(\sft)=F^*(\Phi^{-1}(\sft))\subset\xi_\sft$. By \eqref{eq:F*} we have:
\begin{equation}\label{eq:F**}
\text{$\|g^*(x)-g(x)\|_q<\min\{\mu_k/4,\delta_k/2\}\,$ for each $k\in\N$ and each $x\in X_k\setminus X_{k-1}$.}
\end{equation}
The following commutative diagram summarizes the current situation.
\[
\xymatrix{
X\ar[rr]^f\ar@<0.5ex>[rrdd]^g\ar@<-0.5ex>[rrdd]_{g^*}&&Y\\
\\
|K|\ar[uu]^\Phi_\cong\ar@<0.5ex>[rr]^F\ar@<-0.5ex>[rr]_{F^*}&&|L|\ar[uu]_{\Psi}^{\cong}&\ar@{_{(}->}[l]\ar@{_{(}->}[uul]_{\Psi|_\xi}\xi
}
\]

\noindent{\sc Step III.} \em Construction of the approximating map\em. 
By Lemma \ref{lem:approximation'} there exist $h^*\in\mscr{C}^\infty(X,|L|)$ and for each $\sft\in\Kk$ an open neighborhood $W_\sft\subset X$ of $\sft$ satisfying:
\begin{align}
&\text{$h^*(W_\sft)\subset\xi_\sft\,$ for each $\sft\in\Kk$}\label{eq:g*-h*-2},\\
&\text{$\|h^*(x)-g^*(x)\|_q<\min\{\mu_k/4,\delta_k/2\}\,$ for each $k\in\N$ and each $x\in X_k\setminus X_{k-1}$.}\label{eq:g*-h*-1}
\end{align}
We define $\Hh:=\Psi\circ h^*:X\to Y$ and claim: $\Hh\in\mscr{C}^r(X,Y)$. 

Recall that $\{W_\sft\}_{\sft\in\Kk}$ is an open covering of $X$. By \eqref{eq:g*-h*-2} the restriction $h^*|_{W_\sft}:W_\sft\to\xi_\sft$ is a well-defined $\mscr{C}^\infty$ map for each $\sft\in\Kk$. In addition, $\Hh|_{W_\sft}=\Psi|_{\xi_\sft}\circ h^*|_{W_\sft}$. As both $\Psi|_{\xi_\sft}$ and $h^*|_{W_\sft}$ are $\mscr{C}^r$ maps, $\Hh|_{W_\sft}$ is also a $\mscr{C}^r$ map. Consequently, $\Hh\in\mscr{C}^r(X,Y)$, as claimed.

Next, by \eqref{eq:F**} and \eqref{eq:g*-h*-1} we have
\begin{equation}\label{final}
\text{$\|h^*(x)-g(x)\|_q<\min\{\mu_k/2,\delta_k\}\,$ for each $k\in\N$ and each $x\in X_k\setminus X_{k-1}$.}
\end{equation}
Recall that $g(X_k)=\Psi^{-1}(f(X_k))=Q_k$, so by \eqref{eq:Vk} and \eqref{final} we have $h^*(x)\in V_k$ for each $x\in X_k\setminus X_{k-1}$. Thus, by \eqref{eq:uc-Psi} and \eqref{final} we conclude
\[
\|\Hh(x)-f(x)\|_n=\|\Psi(h^*(x))-\Psi(g(x))\|_n<\epsilon_k\leq\veps(x),
\]
for each $x\in X_k\setminus X_{k-1}$ and each $k\in\N$. Thus, $\|\Hh(x)-f(x)\|_n<\veps(x)$ for each $x\in X$, as required.\qed

%%%
\subsection{Proof of Theorems \ref{thm:relative-locally-constant}}\label{subsec:rlc}
As the pair $(X,X')$ is weakly$^*$ $\mscr{C}^r$ triangulable, there exists a locally finite simplicial complex $K$, a subcomplex $K'$ of $K$ and a homeomorphism $\Phi:|K|\to X$ such that $\Phi(|K'|)=X'$. We repeat the preceding proof of Theorem \ref{thm:main1} with the following changes:
\begin{itemize}
\item We refine the triangulation $\Psi:|L|\to Y$ in such a way that each connected component of $f(X')$ is a vertex $w_k$ of $L$. Denote $X'_k:=f^{-1}(\{w_k\})$ and observe that $X_k'$ is a union of some connected components of $X'$. Let $K'_k$ be the subcomplex of $K'$ such that $\Phi(|K'_k|)=X_k'$.
\item Replace {\sc Step I} with the last two sentences of such step. Namely: `Set $P:=|K|$, $\Kk:=\{\Phi(\sigma)\}_{\sigma\in K}$ and $\Kk^0:=\{\Phi(\sigma^0)\}_{\sigma\in K}$. Define $Q:=|L|\subset\R^q$.'.
\item In {\sc Step II} we apply Remark \ref{rem:rel}(i) to $F:=\Psi^{-1}\circ f\circ\Phi$ with $H:=K'$. The reader should have in mind that $F|_{|K'_k|}$ is constantly equal to the vertex $w_k$ of $L$, so $F$ is simplicial on $K'_k$. Thus, we obtain $g^*:X\to|L|$ such that $g^*(x)=g(x)=w_k$ for each $x\in X'_k$.
\end{itemize}
Finally, by Remark \ref{rem:constant1} the $\mscr{C}^r$ map $h^*:X\to|L|$ that approximates $g^*$ is constantly equal to $w_k$ on $X'_k$. Consequently, the $\mscr{C}^r$ map $\Hh:X\to Y$ that approximates $f$ satisfies $\Hh(x)=\Psi(w_k)=f(x)$ for each $x\in X'_k$, so $\Hh|_{X'}=f|_{X'}$, as required.
\qed

%%%
\subsection{Proof of Theorem \ref{thm:relative-discrete}}\label{subsec:rd} 
Proceeding as in the \textsc{Step I} of the proof of Theorem \ref{thm:main1} we can assume that $X$ is an open subset of $\R^m$. Thus, there exist a locally finite simplicial complex $K$ of some $\R^p$ and a homeomorphism $\Phi:|K|\to X$ such that the restriction $\Phi|_{\sigma^0}:\sigma^0\to X$ is a $\mscr{C}^\infty$ map for each open simplex $\sigma^0$ of $K$. Next, we refine the triangulation $\Phi$ in such a way that there exists a subcomplex $K'$ of $K$ satisfying $\Phi(|K'|)=X'$. This proves that $(X,X')$ is a weakly$^*$ $\mscr{C}^\infty$ triangulable pair. Now, we apply Theorem \ref{thm:relative-locally-constant} to complete the proof.
\qed

%%%
\subsection{Proof of Corollary \ref{cor:KP}}\label{co14}
Let $K$ and $P\subset\R^p$ satisfy the conditions in the statement and let $\veps:P\to\R^+$ be a strictly positive continuous function. Apply Lemma \ref{lem:approximation'} to $X:=P$, $\Phi:=\id_P$, $L:=K$, $g:=\id_P$ and $\delta:=\veps$. We obtain a map $\iota_\veps\in\mscr{C}^\infty(P,P)$ and for each $\sigma\in K$ an open neighborhood $W_\sigma\subset P$ of $\sigma$ such that: 
\begin{itemize}
\item $\iota_\veps(W_\sigma)\subset\sigma$ for each $\sigma\in K$ and
\item $\|\iota_\veps(x)-x\|_p<\veps(x)$ for each $x\in X$.
\end{itemize}
Thus, the net $\{\iota_\veps\}_{\veps\in\mscr{C}^0(P,\R^+)}$ converges to the identity map in $\mscr{C}^0(P,P)$. Consequently, by Lemma \ref{lem:approx} if $f\in\mscr{C}^0(P,Y)$, the net $\{f\circ\iota_\veps\}_{\veps\in\mscr{C}^0(P,\R^+)}$ converges to $f$ in $\mscr{C}^0(P,Y)$. In addition, if $f|_\sigma\in\mscr{C}^r(\sigma,Y)$ for each $\sigma\in K$, every composition $f\circ\iota_\veps:P\to Y$ is a $\mscr{C}^r$ map, because so is the restriction $(f\circ\iota_\veps)|_{W_\sigma}=f|_\sigma\circ\iota_\veps|_{W_\sigma}$ for each $\sigma\in K$. \qed

\begin{remark}
If $K$ is compact, the family of strictly positive constant functions $\epsilon_n:=2^{-n}$ is cofinal in $\mscr{C}^0(P,\R^+)$ and it is enough to construct (using again Lemma \ref{lem:approximation'}) for each $n\in\N$ a map $\iota_n\in\mscr{C}^\infty(P,P)$ and for each $\sigma\in K$ an open neighborhood $W_\sigma\subset P$ of $\sigma$ such that: 
\begin{itemize}
\item $\iota_n(W_\sigma)\subset\sigma$ for each $\sigma\in K$ and
\item $\|\iota_n(x)-x\|_p<\veps_n(x):=\epsilon_n$ for each $x\in X$.
\end{itemize}
Once this is done one proceeds as above. $\sqbullet$
\end{remark}

%%%%%%%

\section{Proof of Theorem \ref{thm:main2}}\label{s4}

In this section we develop first all the machinery we need to prove Theorem \ref{thm:main2}, which is inspired by some techniques contained in \cite{br}: 
\begin{itemize}
\item the construction of $\mscr{C}^\infty$ weak retractions for an analytic normal-crossings divisor $X$ of a real analytic manifold $M$ (that appears in \S\ref{cwr}),
\item immersion of $C$-analytic sets as singular sets of coherent $C$-analytic sets homeomorphic to Euclidean spaces (that appears in \S\ref{icas}),
\end{itemize}
and after we approach its proof (see \S\ref{thm2}).

A weaker and purely semialgebraic version of the arguments used in this section is contained in our manuscript \cite{fg}.

%%%%%%%
\subsection{$\mscr{C}^\infty$ weak retractions}\label{cwr}

In this subsection we construct $\mscr{C}^\infty$ weak retractions $\rho:W\to X$ of open neighborhoods $W$ of an analytic normal-crossings divisor $X$ of a real analytic manifold~$M$ (Proposition \ref{wr}). $\mscr{C}^\infty$ weak retractions $\rho:W\to X$ are $\mscr{C}^\infty$ maps that are arbitrarily close to the identity $\id_X$ on $X$ in the strong $\mscr{C}^0$ topology. As we have already commented, if $X$ is not a $\mscr{C}^\infty$ manifold, we cannot expect that $\rho$ is a retraction onto $X$, that is, there is no hope to have $\rho|_X=\id_X$. Of course, $\rho|_X$ cannot be either a homeomorphism. As the set of proper maps $X\to X$ is open in $\mscr{C}^0(X,X)$, if $\rho:W\to X$ is a $\mscr{C}^\infty$ weak retraction, the restriction $\rho|_X$ is also proper, so fails to be bijective.

Let $M\subset\R^m$ be a $d$-dimensional real analytic manifold and let $X$ be a $C$-analytic subset of $M$. We say that $X$ is an \em analytic normal-crossings divisor of $M$ \em if: 
\begin{itemize}
\item for each point $x\in X$ there exists an open neighborhood $U\subset M$ of $x$ and a real analytic diffeomorphism $\varphi:U\to\R^d$ such that $\varphi(x)=0$ and $\varphi(X\cap U)=\{x_1\cdots x_r=0\}$ for some $r\in\{1,\ldots,d\}$ and 
\item the ($C$-analytic) irreducible components \cite{wb} of (the $C$-analytic set) $X$ are non-singular analytic hypersurfaces of $M$. 
\end{itemize}

In the next result we establish the existence of $\mscr{C}^\infty$ weak retractions $\rho:W\to X$.

\begin{prop}[$\mscr{C}^\infty$ weak retractions]\label{wr}
Let $X$ be an analytic normal-crossings divisor of a real analytic manifold $M$ and let ${\mathcal U}$ be an open neighborhood of $\id_X$ in $\mscr{C}^0(X,X)$. Then there exist an open neighborhood $W$ of $X$ in $M$ and a $\mscr{C}^\infty$ map $\rho:W\to X$ such that $\rho|_X\in{\mathcal U}$.
\end{prop}
\begin{proof}
Assume that $M$ is a real analytic submanifold of some $\R^m$. Choose a strictly positive continuous function $\veps:X\to\R^+$ such that
\begin{equation}\label{eq:1}
\{g\in\mscr{C}^0(X,X):\ \|g(x)-x\|_m<\veps(x) \;\, \forall x\in X\}\subset\mc{U}.
\end{equation}
As $X$ is closed in $M$, we can extend by Tietze's extension theorem $\veps$ to a positive continuous function on $M$ that we denote again $\veps$.

Let $\{X_j\}_{j\in J}$ be the family of the irreducible components of $X$ (see \cite{wb}). Such a family is locally finite in $M$, so $J$ is countable and we assume $J=\N$. If $J$ is finite, the proof is similar but easier.

For each $j\in\N$ denote $\pi_j:\EE_j\to X_j$ the normal bundle of $X_j$ in $M$, where $\EE_j\subset X_j\times\R^m\subset\R^m\times\R^m=\R^{2m}$. Proceeding as the authors do in the proof of \cite[Lem.2.5]{br} one shows that there exists a $\mscr{C}^\infty$ tubular neighborhood map $\phi_j:\EE_j\hookrightarrow M$ of $X_j$ compatible with the other $X_k$ in the following sense: \em for each $x\in\EE_j$ and each $k\in\N\setminus\{j\}$ the image $\phi_j(x)\in X_k$ if and only if $\phi_j(\pi_j(x))\in X_k$\em. Define $\Omega_j:=\phi_j(\EE_j)$ for each $j\in\N$. By Lemma \ref{neighs} we may assume that the family $\{\Omega_j\}_{j\in \N}$ is locally finite in $M$.

Fix $j\in\N$ and let $\eta_j:X_j\to\R^+$ be a strictly positive $\mscr{C}^\infty$ function. Choose a $\mscr{C}^\infty$ function $f:\R\to[0,1]$ such that $f(t)=0$ if $|t|\leq1/4$ and $f(t)=1$ if $|t|\geq1$. Consider the map
\[
h_j:\EE_j\to\EE_j,\ (x,w)\mapsto(x,f(\|w\|_m^2/\eta_j^2(x))w)
\]
and the composition $\psi_j:=\phi_j\circ h_j\circ\phi_j^{-1}:\Omega_j\to\Omega_j$, which is a $\mscr{C}^\infty$ map (to guarantee this fact it is essential to use $\|\cdot\|_m^2$ instead of merely $\|\cdot\|_m$) that extends by the identity to a $\mscr{C}^\infty$ map $\Psi_j:M\to M$. Denote 
\begin{align*}
W_j&:=\phi_j\big(\{(x,w)\in\EE_j:\, \|w\|_m<\eta_j(x)\}\big)\subset\Omega_j,\\
W_j^*&:=\phi_j\big(\{(x,w)\in\EE_j:\, \|w\|_m<\eta_j(x)/2\}\big)\subset W_j.
\end{align*}
We have:
\begin{itemize}
\item $\Psi_j(W_j^*)=X_j$.
\item $\Psi_j(y)=y$ for each $y\in M\setminus W_j$.
\item $\Psi_j(X_k)\subset X_k$ for each $k$.
\item $\Psi_j$ is arbitrarily close to the identity on $M$ if $\eta_j$ is small enough.
\end{itemize}
Only the last assertion requires a further comment. As $\Psi_j$ is the identity on $M\setminus W_j$ and $\phi_j$ is a real analytic embedding (in particular a proper map onto its image), by Lemma \ref{lem:approx} it is enough to show that $h_j$ can be chosen arbitrarily close to $\id_{\EE_j}$. Pick $(x,w)\in\EE_j$. We have:
\begin{align*}
&h_j(x,w)-(x,w)=(0,(f(\|w\|_m/\eta_j(x))-1)w),\\
&|f(\|w\|_m/\eta_j(x))-1|\begin{cases}
\leq1&\text{if $\|w\|_m<\eta_j(x)$,}\\ 
=0&\text{if $\|w\|_m\geq\eta_j(x)$.}
\end{cases}
\end{align*}
Consequently,
\[
\|h_j(x,w)-(x,w)\|_{2m}=|f(\|w\|_m/\eta_j(x))-1|\|w\|_m<\eta_j(x),
\]
so $h_j$ is arbitrarily close to $\id_{\EE_j}$ and $\Psi_j$ is arbitrarily close to $\id_M$, provided $\eta_j$ is small enough. In particular, the restriction $\Psi_j|_X:X\to X$ is arbitrarily close to $\id_X$.

We claim: \em If the functions $\eta_j$ are small enough, the countable composition $\rho:M\to M$, $\rho:=\cdots\circ\Psi_j\circ\cdots\circ\Psi_0$ is a well-defined $\mscr{C}^\infty$ map and the restriction $\rho|_X:X\to X$ belongs to $\mc{U}$\em. 

As $M$ is Hausdorff, second countable and locally compact (hence also paracompact) and the family $\{\Omega_j\}_{j\in\N}$ is locally finite in $M$, there exists an open covering $\{U_\ell\}_{\ell\in\N}$ of $M$ such that each closure $\cl_M(U_\ell)$ of $U_\ell$ in $M$ is compact and only meets finitely many $\Omega_j$ and the family $\{\cl_M(U_\ell)\}_{\ell\in\N}$ is locally finite in $M$. Let $\{V_\ell\}_{\ell\in\N}$ be a shrinking of $\{U_\ell\}_{\ell\in\N}$ that is an open covering of $M$ and satisfies $K_\ell:=\cl_M(V_\ell)\subset U_\ell$ for each $\ell\in\N$. Denote $s_\ell\in\N$ the cardinality of the set of all $j\in\N$ such that $\Omega_j\cap\cl_M(U_\ell)\neq\varnothing$ for each $\ell\in\N$. Note that $\dist_{\R^m}(K_\ell,M\setminus U_\ell)>0$ and pick $\veps_\ell\in\R$ with $0<\veps_\ell<\dist_{\R^m}(K_\ell,M\setminus U_\ell)$. Bearing in mind Remark \ref{rem:topologies}, for each $j\in\N$ we choose $\eta_j$ small enough to have 
\[
\|\Psi_j(x)-x\|_m<\frac{\veps_\ell}{s_\ell+1}\ \text{ for each $\ell\in\N$ and each $x\in\cl_M(U_\ell)$.}
\]

Fix $\ell\in\N$ with $s_\ell>0$. Write $\{j\in\N:\, \Omega_j\cap V_\ell\neq\varnothing\}=\{j_1,\ldots,j_{s_\ell}\}$ and assume $j_1<\ldots<j_{s_\ell}$. Let us check: \em $\|(\Psi_{j_k}\circ\cdots\circ\Psi_{j_1})(y)-y\|_m<\frac{k\veps_\ell}{s_\ell+1}$ for each $y\in K_\ell$ and each $k\in\{1,\ldots,s_\ell\}$. In particular, 
\begin{equation}\label{clue-1}
(\Psi_{j_k}\circ\cdots\circ\Psi_{j_1})(y)\in U_\ell
\end{equation} 
for each $k\in\{1,\ldots,s_\ell\}$\em.

We proceed by induction on $k$. If $k=1$ the result is true by construction. Assume the result true for $k-1$ and let us check that it is also true for $k$. Pick a point $y\in K_\ell$. As $\|(\Psi_{j_{k-1}}\circ\cdots\circ\Psi_{j_1})(y)-y\|_m<\frac{(k-1)\veps_\ell}{s_\ell+1}<\veps_\ell$, we have $(\Psi_{j_{k-1}}\circ\cdots\circ\Psi_{j_1})(y)\in U_\ell$, so 
\[
\|\Psi_{j_k}((\Psi_{j_{k-1}}\circ\cdots\circ\Psi_{j_1})(y))-(\Psi_{j_{k-1}}\circ\cdots\circ\Psi_{j_1})(y)\|_m<\frac{\veps_\ell}{s_\ell+1}.
\]
Thus, we deduce
\begin{multline*}\label{veps}
\|(\Psi_{j_k}\circ\cdots\circ\Psi_{j_1})(y)-y\|_m\leq\|\Psi_{j_k}((\Psi_{j_{k-1}}\circ\cdots\circ\Psi_{j_1})(y))-(\Psi_{j_{k-1}}\circ\cdots\circ\Psi_{j_1})(y)\|_m\\
+\|(\Psi_{j_{k-1}}\circ\cdots\circ\Psi_{j_1})(y)-y\|_m<\frac{\veps_\ell}{s_\ell+1}+\frac{(k-1)\veps_\ell}{s_\ell+1}=\frac{k\veps_\ell}{s_\ell+1}.
\end{multline*}

By \eqref{clue-1} and the fact that $\Psi_j|_{M\setminus\Omega_j}=\id_{M\setminus\Omega_j}$, we have
\[
(\Psi_{j_s}\circ\cdots\circ\Psi_0)(y)=(\Psi_{j_s}\circ\cdots\circ\Psi_{j_1})(y)\in U_\ell\, \text{ for each $y\in K_\ell$}.
\]
As $U_\ell\subset\bigcap_{j>j_s}(M\setminus\Omega_j)$, we conclude
\[
\rho(y)=(\Psi_{j_s}\circ\cdots\circ\Psi_0)(y)=(\Psi_{j_s}\circ\cdots\circ\Psi_{j_1})(y).
\]
It follows that
\begin{equation*}
\text{$\|\rho(y)-y\|_m<\veps_\ell\,$ \text{ for each $\ell\in\N$ and each $y\in K_\ell=\cl_M(V_\ell)$.}}
\end{equation*}
In particular, as $\{V_\ell\}_{\ell\in\N}$ is an open covering of $M$, the composition $\rho$ turns out to be a well-defined $\mscr{C}^\infty$ map, which is arbitrarily close to $\id_M$ in $\mscr{C}^0(M,M)$ if the values $\veps_\ell$ are chosen small enough. We may assume in addition
\begin{equation}\label{eq:2}
\|\rho(y)-y\|_m<\veps(y)\, \text{ for each $y\in M$.}
\end{equation}

We claim: \em $\rho$ maps the open neighborhood
\[
W:=\bigcup_{j\in\N}(\Psi_{j-1}\circ\cdots\circ\Psi_0)^{-1}(W_j^*)\subset M
\]
of $X$ onto $X$\em, where $\Psi_{j-1}\circ\cdots\circ\Psi_0$ denotes $\id_M$ if $j=0$. 

Indeed, pick $y\in W$ and let $j\in\N$ be such that $y\in(\Psi_{j-1}\circ\cdots\circ\Psi_0)^{-1}(W_j^*)$. Define $z\in W_j^*$ by $z:=(\Psi_{j-1}\circ\cdots\circ\Psi_0)(y)$. We have
\[
\rho(y)=\big((\cdots\circ\Psi_{j+1})\circ\Psi_j \circ (\Psi_{j-1}\circ\cdots\circ\Psi_0)\big)(y)=(\cdots\circ\Psi_{j+1})(\Psi_j(z)).
\]
As $\Psi_j(z)\in\Psi_j(W_j^*)=X_j$ and $\Psi_k(X_j)\subset X_j$ for each $k$, we have $\rho(y)=(\cdots\circ\Psi_{j+1})(\Psi_j(z))\in X_j\subset X$, so we conclude $\rho(W)\subset X$. Thus, the corresponding restriction $\rho:W\to X$ is a well-defined $\mscr{C}^\infty$ map. By \eqref{eq:1} and \eqref{eq:2} the restriction $\rho|_X$ belongs to $\mc{U}$, as required.
\end{proof}

\subsection{Immersions of $C$-analytic sets as singular sets}\label{icas}
The following result is a $C$-analytic version of Lemma 2.2 in \cite{br}, which is crucial for the proof of Theorem \ref{thm:main2}. Recall that a $C$-analytic set $Y\subset\R^n$ is \em coherent \em if the ideal $\J_y$ (that is, the stalk of the sheaf of ideals $\J:=\J(Y)\an_{\R^n}$ at $y$) coincides with the ideal of germs of real analytic functions on $\R^n$ whose zero sets contain the germ $Y_y$ for each $y\in Y$.

\begin{lem}[$C$-analytic sets as singular sets]\label{y}
Let $Y$ be a $C$-analytic subset of $\R^n$. Denote $(x,y_1,y_2)$ the coordinates of $\R^n\times\R\times\R=\R^{n+2}$. Then there exists an irreducible coherent $C$-analytic subset $Z$ of $\R^{n+2}$ such that $\Sing(Z)=Y\times\{(0,0)\}$ and the restriction to $Z$ of the projection $\pi:\R^{n+2}\to\R^{n+1}$, $(x,y_1,y_2)\mapsto(x,y_1)$ is a homeomorphism. In addition, the restriction $\pi|_{Z\setminus(Y\times\{(0,0)\})}:Z\setminus(Y\times\{(0,0)\})\to\R^{n+1}\setminus(Y\times\{0\})$ is an analytic diffeomorphism.
\end{lem}
\begin{proof}
Let $f\in\an(\R^n)$ be a global analytic equation of $Y$ and consider the real analytic function $g(x,y_1,y_2):=f(x)^2+y_1^2-y_2^3\in\an(\R^{n+2})$. Define $Z:=\{g=0\}$. Given $(x,y_1)\in\R^{n+1}$, the formula $y_2=(f(x)^2+y_1^2)^{1/3}$ provides the unique solution to the equation $g(x,y_1,y_2)=0$. Hence, $\pi|_Z:Z\to\R^{n+1}$ is a homeomorphism and the restriction $\pi|_{Z\setminus(Y\times\{(0,0)\})}:Z\setminus(Y\times\{(0,0)\})\to\R^{n+1}\setminus(Y\times\{0\})$ is an analytic diffeomorphism.

Let $p:=(p_0,p_1,p_2)\in Z$. If $p\not\in Z\cap\{f(x)=0,y_1=0\}=Y\times\{(0,0)\}$, then it is a regular point of $Z$, because $\frac{\partial{g}}{\partial y_2}(p)\neq0$. As a consequence, the germ $Z_p$ is irreducible and coherent.

Suppose now $p=(p_0,0,0)\in Y\times\{(0,0)\}$. Let us prove: \em The ideal $\J(Z_p)$ of analytic germs vanishing identically on $Z_p$ is generated by $g$\em. Consequently, $Z_p$ is coherent and
\[
\Sing(Z)=Z\cap\{\nabla g(x,y_1,y_2)=0\}=\{f(x)=0,y_1=0,y_2=0\}=Y\times\{(0,0)\}.
\]

After a translation we may assume $p=(0,0,0)$ and we change $f$ by $f(x+p_0)$ (but keep the notation $f$ to simplify notation). As $f(0)=0$, the convergent series $g$ is a distinguished polynomial of degree $2$ with respect to the variable $y_1$. Pick $h\in\J(Z_p)$ and divide it by $g$ using R\"uckert division theorem. There exist analytic series $q\in\R\{x,y_1,y_2\}$ and $a,b\in\R\{x,y_2\}$ such that
\begin{equation}\label{f1}
h(x,y_1,y_2)=q(x,y_1,y_2)g(x,y_1,y_2)+a(x,y_2)y_1+b(x,y_2).
\end{equation}
Changing $y_1$ by $-y_1$ we obtain 
\begin{equation}\label{f2}
h(x,-y_1,y_2)=q(x,-y_1,y_2)g(x,y_1,y_2)-a(x,y_2)y_1+b(x,y_2).
\end{equation}
Adding equations \eqref{f1} and \eqref{f2}, we obtain 
\[
h(x,y_1,y_2)+h(x,-y_1,y_2)=\big(q(x,y_1,y_2)+q(x,-y_1,y_2)\big)g(x,y_1,y_2)+2b(x,y_2).
\]
Observe that $Z$ is symmetric with respect to the variable $y_1$, that is, $(x,y_1,y_2)\in Z$ if and only if $(x,-y_1,y_2)\in Z$. Consequently, $h(x,y_1,y_2)+h(x,-y_1,y_2)\in\J(Z_p)$ and we deduce that $b(x,y_2)\in\J(Z_p)$. Assume by contradiction that $b(x,y_2)\neq0$. Then $b(x,y_2)\in(\J(Z_p)\cap\R\{x,y_2\})\setminus\{0\}$ and by \cite[II.2.3]{rz} the ideal $\J(Z_p)$ has height $\geq2$. This means that the dimension of the germ $Z_p$ is $\leq n+2-2=n$, which is a contradiction because, as $Z$ is homeomorphic to $\R^{n+1}$, we have $\dim(Z_p)=n+1$. Thus, $b=0$ and
\[
h(x,y_1,y_2)=q(x,y_1,y_2)g(x,y_1,y_2)+a(x,y_2)y_1.
\]
As $h,g\in\J(Z_p)$, we have $a(x,y_2)y_1\in\J(Z_p)$. Assume by contradiction that $a(x,y_2)\neq0$. Then 
\[
a(x,y_2)^2f(x)^2-a(x,y_2)^2y_2^3=a(x,y_2)^2g(x,y_1,y_2)-a(x,y_2)^2y_1^2\in(\J(Z_p)\cap\R\{x,y_2\})\setminus\{0\}.
\]
Analogously to what we have inferred from the assumption $b\neq0$, we also achieve a contradiction in this case. We conclude that $h=qg$. Thus, $\J(Z_p)=g\R\{x,y_1,y_2\}$, as claimed.

To finish we have to prove: \em $g$ is irreducible in $\R\{x,y_1,y_2\}$\em. This means that the ideal $\J(Z_p)$ is prime and the analytic germ $Z_p$ is irreducible. Note that local irreducibility (at each point $p\in Z$) implies global irreducibility. 

As $g$ is a distinguished polynomial with respect to $y_1$, it is enough to prove the irreducibility of $g$ in $\R\{x,y_2\}[y_1]$. As $g$ is a monic polynomial with respect to $y_1$, if it is reducible, there exists polynomials of degree one $y_1+a_1,y_1+a_2\in\R\{x,y_2\}[y_1]$ such that $g=(y_1+a_1)(y_1+a_2)$. If we make $x=0$, we have 
\begin{align*}
y_1^2-y_2^3&=g(0,y_1,y_2)=(y_1+a_1(0,y_2))(y_1+a_2(0,y_2))\\
&=y_1^2+(a_1(0,y_2)+a_2(0,y_2))y_1+a_1(0,y_2)a_2(0,y_2),
\end{align*}
so $a_2(0,y_2)=-a_1(0,y_2)$ and $y_2^3=a_1(0,y_2)^2$, which is a contradiction. Consequently, $g$ is irreducible in $\R\{x,y_1,y_2\}$, as required. 
\end{proof}

\subsection{Proof of Theorem \ref{thm:main2}}\label{thm2}
Let $X\subset\R^m$ be a locally compact set and let $Y$ be a $C$-analytic set. We must prove that $\mscr{C}^\infty_*(X,Y)$ is dense in $\mscr{C}^0_*(X,Y)$. By the second part of Corollary \ref{cor:closedness}, we can assume $X$ is closed in $\R^m$ and $Y$ is a $C$-analytic subset of some $\R^n$. We assume $X$ is non-compact. If $X$ is compact the proof is similar, but easier. Denote $\pi:\R^{n+2}\to\R^{n+1}$ the projection onto the first $n+1$ coordinates. By Lemma \ref{y} there exists an irreducible coherent $C$-analytic subset $Z$ of $\R^{n+2}$ such that $\Sing(Z)=Y\times\{(0,0)\}\subset\{x_{n+1}=0,x_{n+2}=0\}$ and the restriction $\psi:=\pi|_{Z}:Z\to\R^{n+1}$ is a homeomorphism. 

As $Z$ is a coherent analytic subset of $\R^{n+1}$, then the pair $(Z,\an_{\R^{n+2}}|_Z)$ (with the analytic structure induced by the one of $\R^{n+2}$) is a (coherent) real analytic space. By \cite[\S13]{bm2} there exist a real analytic manifold $Z'\subset\R^q$ and a proper real analytic map $\phi:Z'\to Z$ such that the restriction
\[
\phi|_{Z'\setminus\phi^{-1}(\Sing(Z))}:Z'\setminus\phi^{-1}(\Sing(Z))\to Z\setminus\Sing(Z)
\]
is a real analytic diffeomorphism and $Y':=\phi^{-1}(\Sing(Z))=\phi^{-1}(Y\times\{(0,0)\})$ is an analytic normal-crossings divisor of $Z'$.

Let $f\in\mscr{C}^0_*(X,Y)$ and let $\veps:X\to\R^+$ be a strictly positive continuous function. As $X$ is non-compact and $f$ is proper, $f(X)$ is unbounded in $\R^n$. In this way, there exists an exhaustion $\{L_\ell\}_{\ell\in\N}$ of $\R^{n+1}$ by compact sets such that $L_{\ell-1}\subsetneq\Int_{\R^{n+1}}(L_\ell)$ and $(L_\ell\setminus L_{\ell-1})\cap(f(X)\times\{0\})\neq\varnothing$ for each $\ell\in\N$, where $L_{-1}:=\varnothing$. As $Y$ is closed in $\R^{n+1}$ and $f$ is proper, $K_\ell:=(f,0)^{-1}(L_\ell\cap(Y\times\{0\}))$ is a compact subset of $X$ for each $\ell\in\N$. Define the non-empty compact sets $N_\ell:=L_\ell\setminus \Int_{\R^{n+1}}(L_{\ell-1})$ and $H_\ell:=K_\ell\setminus\Int_X(K_{\ell-1})$ for each $\ell\in\N$, where $K_{-1}:=\varnothing$. Note that $X=\bigcup_{\ell\in\N}H_\ell$ and $(f,0)(H_\ell)\subset N_\ell$ for each $\ell\in\N$. Define $\delta_\ell:=\min_{H_\ell}(\veps/6)>0$ and choose a strictly positive continuous function $\delta:\R^{n+1}\to\R^+$ such that $\max_{N_\ell}(\delta)\leq\delta_\ell$ for each $\ell\in\N$ (see Remark \ref{rem:topologies}). Observe that $\delta\circ(f,0):X\to\R$ satisfies
\[
\delta\circ(f,0)\leq\veps/6\,\text{ on $X$}.
\]
 Indeed, $\max_{H_\ell}(\delta\circ(f,0))\leq\max_{N_\ell}(\delta)\leq\delta_\ell=\min_{H_\ell}(\veps/6)$ for each $\ell\in\N$.

By Lemma \ref{lem:approx} the map
\[
\mscr{C}^0(X,\R^{n+1})\to\mscr{C}^0(X,\R),\ g\mapsto\delta\circ g
\]
is continuous. Thus, there exists a strictly positive continuous function $\gamma:X\to\R^+$ such that: if $f_1\in\mscr{C}^0(X,\R^{n+1})$ satisfies $\|f_1-(f,0)\|_{n+1}<\gamma$, then $|\delta\circ f_1-\delta\circ(f,0)|<\veps/6$. In particular,
\begin{equation}\label{detail}
\delta\circ f_1<\veps/3\, \text{ on $X$.}
\end{equation}

Consider the proper surjective map $\psi\circ\phi:Z'\to\R^{n+1}$, which satisfies $(\psi\circ\phi)(Y')=Y\times\{0\}$. Denote $(\psi\circ\phi)':Y'\to Y\times\{0\}$ the restriction of $\psi\circ\phi$ from $Y'$ to $Y\times\{0\}$. Using Lemma \ref{lem:approx} again we deduce that the map
\[
\mscr{C}^0(Y',Y')\to\mscr{C}^0(Y',Y\times\{0\}),\ g\mapsto(\psi\circ\phi)'\circ g
\]
is continuous. Let $\zeta:Y'\to\R^+$ be a strictly positive continuous function such that: if $g\in\mscr{C}^0(Y',Y')$ satisfies $\|g-{\id_{Y'}}\|_q<\zeta$, then $\|(\psi\circ\phi)'\circ g-(\psi\circ\phi)'\circ{\id_{Y'}}\|_{n+1}<\delta\circ(\psi\circ\phi)'$. By Proposition \ref{wr} there exists an open neighborhood $W\subset Z'$ of $Y'$ and a $\mscr{C}^\infty$ weak retraction $\rho:W\to Y'$ such that $\|\rho|_{Y'}-{\id_{Y'}}\|_q<\zeta$, so $\|(\psi\circ\phi)'\circ\rho|_{Y'}-(\psi\circ\phi)'\circ{\id_{Y'}}\|_{n+1}<\delta\circ(\psi\circ\phi)'$. Define the open neighborhood $W'\subset W$ of $Y'$ by setting
\[
W':=\big\{z\in W:\,\|(\psi\circ\phi)(\rho(z))-(\psi\circ\phi)(z)\|_{n+1}<\delta((\psi\circ\phi)(z))\big\}.
\]

Consider the closed subset $C':=Z'\setminus W'$ of $Z'$, which does not meet $Y'=(\psi\circ\phi)^{-1}(Y\times\{0\})$. As $\psi\circ\phi:Z'\to\R^{n+1}$ is proper, $C:=(\psi\circ\phi)(C')$ is a closed subset of $\R^{n+1}$, which does not meet $Y\times\{0\}$. Let $\eta:Y\times\{0\}\to\R^+$ and $\eta':X\to\R^+$ be the strictly positive continuous functions given by
\[
\eta(y,0):=\dist_{\R^{n+1}}((y,0),C)/2\,\text{ if $y\in Y$ and }\, \eta':=\eta\circ (f,0).
\]
Define the strictly positive continuous function $\xi:X\to\R^+$ as $\xi:=\min\{\gamma,\eta',\veps/3\}/2$. The map $f':=(f,\xi):X\to\R^{n+1}$ satisfies 
\begin{equation}\label{uco0}
\|f'-(f,0)\|_{n+1}=\xi<\min\{\gamma,\eta',\veps/3\}\leq\frac{\veps}{3} 
\end{equation}
and $\delta\circ f'<\veps/3$ on the whole $X$ (see \eqref{detail}). Consider the continuous function 
\begin{multline*}
f'':=(\phi|_{Z'\setminus Y'})^{-1}\circ\psi^{-1}|_{\R^{n+1}\setminus\{x_{n+1}=0\}}\circ f':
\\X\to\R^{n+1}\setminus\{x_{n+1}=0\}\to Z\setminus(\{Y\}\times\{(0,0)\})\to Z'\setminus Y'
\end{multline*}
and observe that $f'(x)=(\psi\circ\phi)(f''(x))$ for each $x\in X$.

We claim: $f''(X)\subset W'$. 

If $x\in X$, then $(\psi\circ\phi)(f''(x))=(f(x),\xi(x))$. Thus, 
\[
\|(\psi\circ\phi)(f''(x))-(f(x),0)\|_{n+1}=\xi(x)\leq\eta'(x)<\dist_{\R^{n+1}}((f(x),0),C),
\]
so $(\psi\circ\phi)(f''(x))\not\in C$. Consequently, $f''(x)\not\in C'=Z'\setminus W'$, that is, $f''(x)\in W'$.

In addition, we have:\em
\begin{equation}\label{uco3}
\|(\psi\circ\psi)(\rho(f''(x)))-f'(x)\|_{n+1}<\frac{\veps(x)}{3}\, \text{ for each $x\in X$}.
\end{equation}\em
Indeed, pick $x\in X$. As $f''(x)\in W'$ and $(\delta\circ f')(x)<\veps(x)/3$, it holds
\begin{align*}
\|(\psi\circ\psi)(\rho(f''(x)))-f'(x)\|_{n+1}&=\|(\psi\circ\psi)(\rho(f''(x)))-(\psi\circ\psi)(f''(x))\|_{n+1}\\
&<\delta((\psi\circ\phi)(f''(x)))=(\delta\circ f')(x)<\frac{\veps(x)}{3}.
\end{align*}

The following commutative diagram summarizes the situation we have achieved until the moment.
\[
\xymatrix{
W'\ar@<0.6ex>[d]^\rho\ar@{^{(}->}[dr]\\
Y'\ar@{^{(}->}[r]\ar[d]^{\phi}&Z'\ar[d]^{\phi}\\
Y\times\{(0,0)\}=\Sing(Z)\ar[d]^\psi\ar@{^{(}->}[r]&Z\ar@{^{(}->}[r]\ar@<-0.6ex>@{->}[dr]_\psi&\R^{n+2}\ar[d]^\pi\\
Y\times\{0\}\ar@{^{(}->}[rr]&&\R^{n+1}\ar@<-0.6ex>@{->}[ul]_{\psi^{-1}}\\
X\ar[u]_{(f,0)}\ar[rr]_{f':=(f,\xi)}\ar@/^5.5pc/[uuuu]^{f''}&&\R^{n+1}\setminus\{x_{n+1}=0\}\ar@{^{(}->}[u]\ar@/_5.5pc/[uuul]_{\phi|_{Z\setminus\Sing(Z)}^{-1}\circ\psi^{-1}|_{\R^{n+1}\setminus\{x_{n+1}=0\}}}
}
\]

Proceeding as in {\sc Step I} of the proof of Theorem \ref{thm:main1}, we deduce that there exists an open neighborhood $U\subset\R^m$ of $X$ and a continuous extension $F'':U\to Z'$ of $f''$. The inverse image $U_0:=F''^{-1}(W')$ is an open neighborhood of $X$ in $U$ and the restriction, $F''|_{U_0}:U_0\to W'$ is a continuous map between the real analytic manifolds $U_0$ and $W'$. We substitute $U$ by $U_0$ and $F''$ by $F''|_{U_0}$, but we keep the original notation to ease the writing. Let $H_0:U\to W'$ be a real analytic map arbitrarily close to $F''$ in $\mscr{C}^0(U,W')$, which exists by Whitney's approximation theorem. The restriction $h_0:=H_0|_X:X\to W'$ is a real analytic map arbitrarily close to $f''$ in $\mscr{C}^0(X,W')$. Consider the $\mscr{C}^\infty$ map $h:X\to Y$ such that $(h,0):=(\psi\circ\phi)'\circ\rho|_{W'}\circ h_0$. As the map 
\[
\mscr{C}^0(X,W')\to\mscr{C}^0(X,Y),\ g\mapsto(\psi\circ\phi)'\circ\rho|_{W'}\circ g
\] 
is continuous, we may assume 
\begin{equation}\label{uco4}
\|(h(x),0)-(\psi\circ\phi)'(\rho(f''(x)))\|_{n+1}<\frac{\veps(x)}{3}\, \text{ for each $x\in X$}.
\end{equation}
Given any $x\in X$ we deduce by \eqref{uco0}, \eqref{uco3} and \eqref{uco4},
\begin{align*}
\|h(x)&-f(x)\|_n=\|(h(x),0)-(f(x),0)\|_{n+1}\leq\|(h(x),0)-(\psi\circ\phi)'(\rho(f''(x)))\|_{n+1}\\
&+\|(\psi\circ\phi)'(\rho(f''(x)))-f'(x)\|_{n+1}+\|f'(x)-(f(x),0)\|_{n+1}<\frac{\veps(x)}{3}+\frac{\veps(x)}{3}+\frac{\veps(x)}{3}=\veps(x)
\end{align*}

Thus, we have found a $\mscr{C}^\infty$ map $h:X\to Y$ that is arbitrarily close to $f$. In the same vein of \cite[Thm. II.1.5]{hir3} one easily shows that $\mscr{C}^0_*(X,Y)$ is an open subset of $\mscr{C}^0(X,Y)$, so we can assume that $h$ is in addition proper, as required.
\qed

\begin{remarks}
(i) In the preceding proof we have used that the continuous map $f:X\to Y\subset\R^n$, we want to approximate, is proper exactly when we need to find for a strictly positive continuous function $\veps:X\to\R^+$ another strictly positive continuous function $\delta:\R^{n+1}\to\R^+$ such that $\delta\circ(f,0)<\veps/6$. $\sqbullet$

(ii) The techniques we have developed in this section cannot be adapted to guarantee relative approximation. Let us summarize the principal points of the proof of Theorem \ref{thm:main2} (we keep all the notations already introduced there). We begin moving the image of the map $f$ we want to approximate outside $Y$, but inside a small neighborhood of $Y$ in $\R^{n+1}$ (we have constructed the map $f'$). {\em This movement is crucial to make the rest of our construction work but keeps off relative approximation}. 

Recall that $Y\times\{(0,0)\}$ is the singular set of an irreducible coherent $C$-analytic subset $Z$ of $\R^{n+2}$, which is in addition homeomorphic to $\R^{n+1}$ (via the projection of $\R^{n+2}$ to $\R^{n+1}$ that forgets the last coordinate). Once this is done, we lift the image of $f'$ to $Z\setminus\Sing(Z)$. Next, we use resolution of singularities to change $Y\times\{(0,0)\}=\Sing(Z)$ by an analytic normal-crossings divisor $Y'$ contained in a real analytic manifold $Z'$. We denote $\phi:Z'\to Z$ the (proper) resolution map and consider $f''$ the composition of the lift of $f'$ with $(\phi|_{Z'\setminus Y'})^{-1}$. The image of $f''$ is contained in $Z'\setminus Y'$. 

We extend $f''$ continuously to an open neighborhood $U_0\subset\R^m$ of $X$ such that its image is contained in a small neighborhood $W'\subset Z'$ of $Y'$ endowed with a $\mscr{C}^\infty$ weak retraction $\rho:W\to Y$. The previous extension $F'':U_0\to W'$ is a continuous map between submanifolds of Euclidean spaces, so here Whitney's approximation theorem works and provides a $\mscr{C}^\infty$ approximation $H_0:U_0\to W'$ close to $F''$. We can even assume that the images of $H_0$ and $F''$ are contained in $W'\subset Y'$. Now, we compose $h_0:=H_0|_X$ with $(\psi\circ\phi)'\circ\rho$ to obtain the $\mscr{C}^\infty$ approximating map $h:X\to Y$ of the continuous map $f:X\to Y$. 

There are some difficulties to achieve relative approximation results:

(1) We have moved the image of $f$ off $Y$ to construct $F''$ and we have lost the control of the restrictions of $f$ to subsets $X'$ of $X$.

(2) In case $Y$ is an analytic normal-crossings divisor of a real analytic manifold $Z$, we can skip the first part of the proof and keep the image of $f$ inside $Y$. Then, we extend $f$ continuously to an open neighborhood $U_0\subset\R^m$ of $X$ such that its image is contained in a small neighborhood $W\subset Z$ of $Y$ endowed with a $\mscr{C}^\infty$ weak retraction $\rho:W\to Y$. But now we have to deal with $\rho$, which is not a true retraction and moves the points of $Y$. Thus, if $X'\subset X$ satisfies that $f(X')$ is not contained in the set of fixed points of $\rho$, then it seems difficult to assure that the restriction to $X'$ of the $\mscr{C}^\infty$ approximation behaves as $f$ on $X'$.

(iii) The previous remark does not mean that in the framework of $C$-analytic sets relative approximation is not possible (see the example below). What we have pointed out is that our techniques are not a good tool to approach relative approximation and new ideas are needed.
\end{remarks}

\begin{example}
Let $X':=[-1,0]\subset X:=[-1,1]\subset\R$ and let $Y:=\{xy=0\}\subset\R^2$. Consider the continuous map 
$$
f:X\to Y,\ t\mapsto\begin{cases}
(t,0)&\text{if $t\in[-1,0]$,}\\
(0,t)&\text{if $t\in[0,1]$,}
\end{cases}
$$
which is $\mscr{C}^\infty$ on $X'$. Fix $\veps>0$ and let $\theta_1,\theta_2:[-1,1]\to[0,1]$ be $\mscr{C}^\infty$ bump functions such that: 
\begin{itemize}
\item $\theta_1|_{[-1,0]}=1$ and $\theta_1|_{[\frac{\veps}{2},1]}=0$.
\item $\theta_2|_{[-1,\frac{\veps}{2}]}=0$ and $\theta_1|_{[\veps,1]}=1$.
\end{itemize}
Define
$$
g:X\to Y,\ t\mapsto\begin{cases}
(t\theta_1(t),0)&\text{if $t\in[-1,\frac{\veps}{2}]$,}\\
(0,t\theta_2(t))&\text{if $t\in[\frac{\veps}{2},1]$,}
\end{cases}
$$
which is a $\mscr{C}^\infty$ function. Observe that $g$ coincides with $f$ outside the interval $[0,\veps]$. We have 
$$
\|f(t)-g(t)\|_2=\begin{cases}
\|(0,t)-(t\theta_1(t),0)\|_2=|t|\sqrt{1+\theta_1^2(t)}<\veps&\text{if $t\in[0,\frac{\veps}{2}]$,}\\
\|(0,t)-(0,t\theta_2(t))\|_2=|t||1-\theta_2(t)|<\veps&\text{if $t\in[\frac{\veps}{2},\veps]$,}
\end{cases}
$$
so $g$ is a $\mscr{C}^\infty$ approximation of $f$ such that $g|_{X'}=f|_{X'}$. $\sqbullet$
\end{example}

\bibliographystyle{amsalpha}

\begin{thebibliography}{vdDM}

\bibitem[ABF]{abf} F. Acquistapace, F. Broglia, J.F. Fernando: On Hilbert's 17th problem and Pfister's multiplicative formulae for the ring of real analytic functions. {\em Ann. Sc. Norm. Super. Pisa Cl. Sci.} (5) {\bf13} (2014), no. 2, 333--369.

\bibitem[Al]{alexander} J. W. Alexander: A proof of the invariance of certain constants of analysis situs. \em Trans. Amer. Math. Soc. \em {\bf16} (1915), no. 2, 148--154.

\bibitem[Ar]{ar} M. Artin: On the solutions of analytic equations. {\it Invent. Math.} {\bf5} (1968), 277--291.

\bibitem[BR]{br} T. Baird, D.A. Ramras: Smoothing maps into algebraic sets and spaces of flat connections. \em Geom. Dedicata \em {\bf174} (2015), 359--374.

\bibitem[BFR]{bfr} E. Baro, J.F. Fernando, J.M. Ruiz: Approximation on analytic sets with monomial singularities. \em Adv. Math. \em {\bf262} (2014), 59--114.

\bibitem[BM1]{bm} E. Bierstone, P. D. Milman: Semianalytic and subanalytic sets. \em Inst. Hautes \'Etudes Sci. Publ. Math. \em No. {\bf67} (1988), 5--42.

\bibitem[BM2]{bm2} E. Bierstone, P. D. Milman: Canonical desingularization in characteristic zero by blowing up the maximum strata of a local invariant. {\em Invent. Math.} {\bf128} (1997), 207--302.

\bibitem[BP]{bp} M. Bilski, A. Parusinski: Approximation of holomorphic maps from Runge domains to affine algebraic varieties. \em J. Lond. Math. Soc. \em (2) {\bf90} (2014), no. 3, 807--826.

\bibitem[BCR]{bcr} J. Bochnak, M. Coste, Roy, M.-F. Roy: Real algebraic geometry. \em Ergebnisse der Mathematik und ihrer Grenzgebiete \em (3), {\bf 36}. Springer-Verlag, Berlin (1998).

\bibitem[B]{b} J.L. Bryant: Piecewise linear topology. {\em Handbook of geometric topology}, 219–259, North-Holland, Amsterdam, (2002).

\bibitem[Ca]{ca} S.-S. Cairns: A simple triangulation method for smooth manifolds. {\em Bull. Amer. Math. Soc.} {\bf67} (1961), no. 4, 389--390. 

\bibitem[C]{c} H. Cartan: Vari\'et\'es analytiques r\'eelles et vari\'et\'es analytiques complexes. {\em Bull. Soc. Math. France} {\bf85} (1957), 77--99.

\bibitem[CRS1]{crs1} M. Coste, J. M. Ruiz, M. Shiota: Approximation in compact Nash manifolds. {\it Amer. J. Math.} {\bf117} (1995), no. 4, 905--927.

\bibitem[CRS2]{crs2} M. Coste, J. M. Ruiz, M. Shiota: Global problems on Nash functions. {\it Rev. Mat. Complut.} {\bf17} (2004), no. 1, 83--115.

\bibitem[CP]{cp} M. Czapla, W. Paw\l ucki: Strict $C^1$-triangulations in o-minimal structures, {\em Topol. Methods Nonlinear Anal.} {\bf52} (2018), no. 2, 739--747. 

\bibitem[DP]{DP} S. Daneri, A. Pratelli: Smooth approximation of bi-Lipschitz orientation-preserving homeomorphisms. {\em Ann. Inst. H. Poincar\'e Anal. Non Lin\'eaire} {\bf 31} (2014), no. 3, 567--589. 

\bibitem[Fe]{fe} J.F. Fernando: On Nash images of Euclidean spaces. {\em Adv. Math.} {\bf 311} (2018), 627-719.

\bibitem[FGR]{fgr} J.F. Fernando, J.M. Gamboa, J.M. Ruiz: Finiteness problems on Nash manifolds and Nash sets. {\em J. Eur. Math. Soc.} (JEMS) {\bf16} (2014), no. 3, 537--570.

\bibitem[FG]{fg} J.F. Fernando, R. Ghiloni: Differentiable approximation of continuous semialgebraic maps, {\em Selecta Math. (N.S.)} {\bf25} (2019), no. 3, Paper No. 46, 30 pp. 

\bibitem[HM]{HM} S. Hencl, C. Mora-Corral: Diffeomorphic approximation of continuous almost everywhere injective Sobolev deformations in the plane. {\em Q. J. Math.} {\bf 66} (2015), no. 4, 1055--1062.

\bibitem[Ha]{hanner} O. Hanner: Some theorems on absolute neighborhood retracts. \em Ark. Mat. \em {\bf1} (1951), 389--408.

\bibitem[Hr]{hardt} R. M. Hardt: Triangulation of subanalytic sets and proper light subanalytic maps. {\em Invent. Math.} {\bf38} (1976/77), no. 3, 207--217.

\bibitem[Hi]{hi1} H. Hironaka: On resolution of singularities (characteristic zero). 1963 {\em Proc. Internat. Congr. Mathematicians} (Stockholm, 1962) pp. 507--521 Inst. Mittag-Leffler, Djursholm. 

\bibitem[H1]{hir1} M.W. Hirsch: On combinatorial submanifolds of differentiable manifolds, {\em Comment. Math. Helv.} {\bf36} (1961), 108--111.

\bibitem[H2]{hir2} M.W. Hirsch: Obstruction theories for smoothing manifolds and maps. {\em Bull. Amer. Math. Soc.} {\bf69} (1963), 352--356. 

\bibitem[H3]{hir3} M.W. Hirsch: Differential topology. Corrected reprint of the 1976 original. {\em Graduate Texts in Mathematics}, {\bf33}. Springer-Verlag, New York (1994).

\bibitem[Hu]{hu} J.F.P. Hudson: Piecewise linear topology. \em University of Chicago Lecture Notes \em prepared with the assistance of J. L. Shaneson and J. Lees W. A. Benjamin, Inc., New York-Amsterdam (1969).

\bibitem[IKO]{IKO} T. Iwaniec, L.V. Kovalev, J. Onninen: Diffeomorphic approximation of Sobolev homeomorphisms. {\em Arch. Ration. Mech. Anal.} {\bf 201} (2011), no. 3, 1047--1067.

\bibitem[Jo]{jo} D. Joyce: On manifolds with corners. {\em Advances in geometric analysis}, 225--258, Adv. Lect. Math. (ALM), {\bf21}, Int. Press, Somerville, MA, (2012). 

\bibitem[Le]{le} L. Lempert: Algebraic approximations in analytic geometry. {\em Invent. Math.} {\bf121} (1995), no. 2, 335--353.

\bibitem[Me1]{me1} R.B. Melrose: Calculus of conormal distributions on manifolds with corners. {\em Int. Math. Res. Not. IMRN} {\bf1992}, no. 3, 51--61.

\bibitem[Me2]{me2} R.B. Melrose: Differential Analysis on Manifolds with Corners. (1996) {\em Unfinished book available at} {\tt http://math.mit.edu/$\sim$rbm}.

\bibitem[Mic]{Mic} P.W. Michor: Topics in differential geometry. Graduate Studies in Mathematics, 93. {\it American Mathematical Society}, {\it Providence}, {\it RI}, (2008).

\bibitem[Mi1]{mi1} J.W. Milnor: On manifolds homeomorphic to the 7-sphere. {\em Ann. of Math.} (2) {\bf 64} (1956), 399--405.

\bibitem[Mi2]{mi2} J.W. Milnor: Two complexes which are homeomorphic but combinatorially distinct. {\em Ann. Math.} {\bf75} (1961), 575--590.

\bibitem[MP]{MP} C. Mora-Corral, A. Pratelli: Approximation of piecewise affine homeomorphisms by diffeomorphisms. {\em J. Geom. Anal.} {\bf 24} (2014), no. 3, 1398--1424.

\bibitem[Mu]{Mu} S. M\"{u}ller: Uniform approximation of homeomorphisms by diffeomorphisms. {\em Topology Appl.} {\bf 178} (2014), 315--319.

\bibitem[M1]{mu1} J.R. Munkres: Obstructions to the smoothing of piecewise differentiable homeomorphisms, {\em Ann. of Math.} (2) {\bf72} (1960), 521--554.

\bibitem[M2]{mu3} J.R. Munkres: Obstructions to extending diffeomorphisms, {\em Proc. Amer. Math. Soc.} {\bf15} (1964) 297--299. 

\bibitem[M3]{mu2} J.R. Munkres: Obstructions to imposing differentiable structures, {\em Illinois J. Math.} {\bf8} (1964), no. 3, 361--376.

\bibitem[M4]{mu4} J.R. Munkres: Elementary differential topology. Lectures given at Massachusetts Institute of Technology, Fall (1961). Revised edition. \em Annals of Mathematics Studies\em, {\bf54} Princeton University Press, Princeton, N.J. (1966).

\bibitem[M5]{mu5} J.R. Munkres: Elements of algebraic topology. {\it Addison-Wesley Publishing Company, Menlo Park, CA,} (1984).

\bibitem[Na]{n} R. Narasimhan: Analysis on real and complex manifolds. {\em Advanced Studies in Pure Mathematics}, {\bf 1} Masson \& Cie, Editeurs, Paris; North-Holland Publishing Co., Amsterdam (1968).

\bibitem[OS]{os} T. Ohmoto, M. Shiota: ${\mathcal C}^1$-triangulations of semialgebraic sets, J. Topol. (3) {\bf10} (2017), 765--775.

\bibitem[ORR]{orr} E. Outerelo, J.A. Rojo, J.M. Ruiz: Topolog\'{\i}a diferencial (2$^a$ edici\'on) un curso de iniciaci\'on. {\em Sanz y Torres, S.L.}, Madrid, (2020).

\bibitem[Rz]{rz} J.M. Ruiz: The basic theory of power series. {\em Advanced Lectures in Mathematics}. Friedr. Vieweg \& Sohn, Braunschweig (1993).

\bibitem[Sh]{sh1} M. Shiota: Geometry of subanalytic and semialgebraic sets. {\em Progress in Mathematics}, {\bf150}. Birkh\"{a}user Boston, Inc., Boston, MA (1997).

\bibitem[Th]{th} R. Thom: Des vari\'et\'es triangul\^ees aux vari\'et\'es diff\'erentiables. {\em Proc. Internat. Congress Math.} (1958) Cambridge Univ. Press, Cambridge, (1960), 248--255.

\bibitem[vD]{vd2} L. van den Dries: Tame topology and o-minimal structures. {\it London Mathematical Society Lecture Note Series}, {\bf248}. Cambridge University Press, Cambridge (1998).

\bibitem[vdDM]{vm} L. van den Dries, C. Miller: Geometric categories and o-minimal structures. {\em Duke Math. J.} {\bf84} (1996), no. 2, 497--540.

\bibitem[W]{wh2} H. Whitney: Differentiable manifolds. {\em Ann. of Math.} (2) {\bf 37} (1936), no. 3, 645--680.

\bibitem[WB]{wb} H. Whitney, F. Bruhat: Quelques propri\'et\'es fondamentales des ensembles analytiques r\'eels. \em Comment. Math. Helv. \em {\bf33} (1959), 132--160. 

\bibitem[Wi]{wi} A.J. Wilkie: Lectures on elimination theory for semialgebraic and subanalytic sets. O-minimality and diophantine geometry, 159--192, {\em London Math. Soc. Lecture Note Ser.}, {\bf421}, Cambridge Univ. Press, Cambridge (2015).

\end{thebibliography}

\end{document}